\documentclass{birkjour}
 \newtheorem{thm}{Theorem}[section]
 \newtheorem{cor}[thm]{Corollary}
 \newtheorem{lem}[thm]{Lemma}
 \newtheorem{prop}[thm]{Proposition}
 \theoremstyle{definition}
 \newtheorem{defn}[thm]{Definition}
 \theoremstyle{remark}
 \newtheorem{rem}[thm]{Remark}

 \usepackage{amsmath}
 \usepackage{amssymb}
 \numberwithin{equation}{section}

\begin{document}

\title[Some interpolation inequalities]
{Some interpolation inequalities in Lorentz, Morrey and BMO spaces}
\author{Hua Wang}
\address{School of Mathematics and Systems Science,\\
Xinjiang University,\\
Urumqi 830046, P. R. China}

\email{wanghua@pku.edu.cn.}



\date{\today}

\dedicatory{In memory of Li Xue.}

\begin{abstract}
In this paper, the author establishes some interpolation results between Lorentz, Morrey and BMO spaces. Let $1<p<\infty$ and $p\leq r\leq\infty$. It is proved that the space $L^{p,r}(\mathbb R^n)\cap\mathrm{BMO}(\mathbb R^n)$ is continuously embedded into $L^q(\mathbb R^n)$ for all $q$ with $p<q<\infty$, where $L^{p,r}(\mathbb R^n)$ denotes the classical Lorentz space with indices $p$ and $r$. Moreover, the author establishes the optimal growth rate of this embedding constant as $q\to\infty$. Based on Morrey spaces, the author introduces a new family of function spaces called Lorentz--Morrey spaces $LM^{p,r;\kappa}(\mathbb R^n)$ with indices $p$, $r$ and $\kappa$, and then shows that the space $LM^{p,r;\kappa}(\mathbb R^n)\cap \mathrm{BMO}(\mathbb R^n)$ is continuously embedded into $L^{q;\kappa}(\mathbb R^n)$ for all $q$ with $p<q<\infty$, where $1<p<\infty$, $p\leq r\leq\infty$ and $0<\kappa<1$. Furthermore, the asymptotically optimal growth order of this embedding constant is also established. As an application of the above interpolation results, some new bilinear estimates in the setting of Lorentz and Lorentz--Morrey spaces are also obtained, which can be used in the study of the global existence and regularity of weak solutions to elliptic and parabolic partial differential equations of the second order.
\end{abstract}

\maketitle

\section{Introduction and preliminaries}
In this paper, we are concerned with the theory of real interpolation in Lorentz, Morrey and BMO spaces. We write $\mathbb{N}=\{1,2,3,\dots\}$ for the set of natural numbers. Let $n\in \mathbb{N}$ and $\mathbb R^n$ be the $n$-dimensional Euclidean space with the Lebesgue measure $dx$. The Euclidean norm of $x\in\mathbb R^n$ is denoted by $|x|$. Throughout this paper, the letter $C$ always denotes a positive constant independent of the main parameters involved, but it may be different from line to line. We use the symbol $\mathbf{A}\lesssim \mathbf{B}$ to denote that there exists a constant $C>0$ such that $\mathbf{A}\leq C\cdot\mathbf{B}$. We also use $\mathbf{A}\approx \mathbf{B}$ to denote the equivalence of $\mathbf{A}$ and $\mathbf{B}$, that is, there exist two positive constants $C_1$, $C_2$ independent of $\mathbf{A}$ and $\mathbf{B}$ such that $C_1\cdot \mathbf{A}\leq \mathbf{B}\leq C_2\cdot\mathbf{A}$. For any given normed spaces $\mathcal{X}$ and $\mathcal{Y}$, with the (quasi-)norms $\|\cdot\|_{\mathcal{X}}$ and $\|\cdot\|_{\mathcal{Y}}$, respectively, the symbol $\mathcal{X}\hookrightarrow\mathcal{Y}$ means that the normed space $\mathcal{X}$ is continuously embedded into $\mathcal{Y}$, that is to say, for any $f\in \mathcal{X}$, then $f\in \mathcal{Y}$ and $\|f\|_{\mathcal{Y}}\lesssim\|f\|_{\mathcal{X}}$ with the implicit positive constant independent of $f$. Let $1<p<\infty$, $p\leq r\leq\infty$ and $L^{p,r}(\mathbb R^n)$ be the classical Lorentz space with indices $p$ and $r$. We shall prove the following interpolation results:
\begin{enumerate}
  \item $L^{p,r}(\mathbb R^n)\cap L^{\infty}(\mathbb R^n)\hookrightarrow L^q(\mathbb R^n)$ for all $q$ with $p<q<\infty$. Moreover, this embedding constant is shown to be 2 when $p\leq r<\infty$(with the asymptotically constant 2 when $r=\infty$).
  \item $L^{p,r}(\mathbb R^n)\cap \mathrm{BMO}(\mathbb R^n)\hookrightarrow L^q(\mathbb R^n)$ for all $q$ with $p<q<\infty$. Moreover, the optimal growth rate of this embedding constant is given, as $q\to\infty$.
\end{enumerate}
Actually, we can prove the second result with $\mathrm{BMO}$ replaced by $\mathcal{W}$, where the function space $\mathcal{W}$ is defined as the set of all locally integrable functions $f$ for which
\begin{equation*}
\|f\|_{\mathcal{W}}=\sup_{t>0}\big[f^{**}(t)-f^{*}(t)\big]<+\infty.
\end{equation*}
Here $f^{*}$ and $f^{**}$ denote the non-increasing rearrangement of $f$ and the average function of $f^{*}$, respectively. The set $\mathcal{W}$ was first introduced by Bennett, DeVore and Sharpley in \cite{be1}(see also Bennett--Sharpley in \cite{be2}, and Kozono--Wadade in \cite{Kozono2}). Moreover, it was also shown in \cite{Kozono2} that the following estimate
\begin{equation*}
\|f\|_{\mathcal{W}}\leq C(n)\|f\|_{\mathrm{BMO}}
\end{equation*}
holds for all $f\in \mathrm{BMO}(\mathbb R^n)$, where $C(n)$ is a positive constant depending only on the dimension $n$. From this, the desired result follows.

Given a Lebesgue measurable set $E\subset\mathbb R^n$, we denote the $n$-dimensional Lebesgue measure of $E$ by $m(E)$ and the
characteristic function of $E$ by $\mathbf{1}_E$. Let $1\leq p<\infty$, $1\leq r\leq\infty$ and $0<\kappa<1$. We introduce a new class of function spaces called Lorentz--Morrey spaces $LM^{p,r;\kappa}(\mathbb R^n)$. We define the functional
\begin{equation*}
\|f\|_{LM^{p,r;\kappa}}=\sup_{B\subset\mathbb R^n}\frac{1}{m(B)^{\kappa/p}}\|f\|_{L^{p,r}(B)},
\end{equation*}
where the supremum is taken over all balls $B$ in $\mathbb R^n$ and
\begin{equation*}
\|f\|_{L^{p,r}(B)}=
\begin{cases}
\displaystyle\bigg(\int_0^{\infty}\Big[s^{1/p}f^{*}_{B}(s)\Big]^r\frac{ds}{s}\bigg)^{1/r}, &\mbox{if}~ 1\leq r<\infty,\\
\displaystyle\sup_{s>0}\Big[s^{1/p}f^{*}_{B}(s)\Big], &\mbox{if}~ r=\infty.
\end{cases}
\end{equation*}
Here $f^{*}_{B}(s)=(f\cdot\mathbf{1}_{B})^{*}(s)$. The Lorentz--Morrey space $LM^{p,r;\kappa}(\mathbb R^n)$ is defined to be the set of all measurable functions $f$ on $\mathbb R^n$ for which $\|f\|_{LM^{p,r;\kappa}}<+\infty$.

Let $1<p<\infty$, $p\leq r\leq\infty$ and $0<\kappa<1$. We shall also show the following interpolation results:
\begin{enumerate}
  \item $LM^{p,r;\kappa}(\mathbb R^n)\cap L^{\infty}(\mathbb R^n)\hookrightarrow L^{q;\kappa}(\mathbb R^n)$ for all $q$ with $p<q<\infty$. Moreover, this embedding constant is equal to 2 when $p\leq r<\infty$(with the asymptotically constant 2 when $r=\infty$).
  \item $LM^{p,r;\kappa}\cap \mathrm{BMO}(\mathbb R^n)\hookrightarrow L^{q;\kappa}(\mathbb R^n)$ for all $q$ with $p<q<\infty$. Moreover, this embedding constant is shown to have the linear growth order as $q\to\infty$.
\end{enumerate}
For any measurable set $E\subseteq\mathbb R^n$ and any given $p\in[1,\infty]$, the Lebesgue space $L^p(E)$ is defined to be the set of all measurable functions $f$ on $E$ such that $\|f\|_{L^p(E)}<+\infty$, where
\begin{equation*}
\|f\|_{L^p(E)}:=
\begin{cases}
\displaystyle\bigg(\int_{E}|f(x)|^p\,dx\bigg)^{1/p}, &\mbox{if}~ p\in[1,\infty),\\
\underset{x\in E}{\mbox{ess\,sup}}\,|f(x)|, &\mbox{if}~ p=\infty.
\end{cases}
\end{equation*}
We also denote by $WL^{p}(E)$($1\leq p<\infty$) the weak Lebesgue space consisting of all measurable functions $f$ on $E$ such that
\begin{equation*}
\|f\|_{WL^{p}(E)}:=
\sup_{\lambda>0}\lambda\cdot\Big[m\big(\big\{x\in E:|f(x)|>\lambda\big\}\big)\Big]^{1/p}<+\infty.
\end{equation*}
When $E=\mathbb R^n$, we write the norms of $L^p(\mathbb R^n)$ and $WL^{p}(\mathbb R^n)$ simply by $\|\cdot\|_{L^p}$ and $\|\cdot\|_{WL^{p}}$, respectively. The space of functions with bounded mean oscillation (BMO) plays a crucial role in modern harmonic analysis, see, for example,
\cite[Chapter 6]{duoand}, \cite[Chapter 3]{Grafakos} and \cite[Chapter 4]{stein2}. A locally integrable function $f$ on $\mathbb R^n$ is said to be in $\mathrm{BMO}(\mathbb R^n)$, if for any ball $B\subset\mathbb R^n$,
\begin{equation*}
\frac{1}{m(B)}\int_{B}|f(x)-f_{B}|\,dx<+\infty,
\end{equation*}
where $f_{B}$ denotes the average value of $f$ on the ball $B$, that is,
\begin{equation*}
f_{B}:=\frac{1}{m(B)}\int_{B}f(y)\,dy.
\end{equation*}
The $\mathrm{BMO}$ norm of $f$ is defined by
\begin{equation*}
\|f\|_{\mathrm{BMO}}:=\sup_{B\subset\mathbb R^n}\frac{1}{m(B)}\int_{B}|f(x)-f_{B}|\,dx,
\end{equation*}
where the supremum is taken over all balls $B$ in $\mathbb R^n$. The space $\mathrm{BMO}(\mathbb R^n)$ was first introduced by John and Nirenberg in \cite{john}.
\begin{rem}
It is well known that for any $1\leq p<\infty$, $L^p(\mathbb R^n)\subset WL^{p}(\mathbb R^n)$, and this inclusion is strict. Let $\mathcal{F}(x):=|x|^{-n/p}$ defined on $\mathbb R^n$ with the usual Lebesgue measure. It is obvious that $\mathcal{F}$ is not in $L^p(\mathbb R^n)$ but $\mathcal{F}$ is in $WL^{p}(\mathbb R^n)$ with $\|\mathcal{F}\|_{WL^{p}}=(\nu_n)^{1/p}$, where $\nu_n$ is the Lebesgue measure of the unit ball in $\mathbb R^n$ (see \cite{Grafakos2}).
\end{rem}

\begin{rem}
It is well known that $L^\infty(\mathbb R^n)\subset \mathrm{BMO}(\mathbb R^n)$, and this inclusion is strict. The typical example of a function that is in $\mathrm{BMO}(\mathbb R^n)$ but not in $L^\infty(\mathbb R^n)$ is $\mathcal{F}(x):=\log |x|$ on $\mathbb R^n$. Furthermore, it is not difficult to see that if $\delta\geq0$, then $(\log|x|)^{\delta}\in \mathrm{BMO}(\mathbb R^n)$ if and only if $\delta\leq1$. The space $\mathrm{BMO}(\mathbb R^n)$ arises naturally as a substitute for the space $L^\infty(\mathbb R^n)$ in the study of boundedness of integral operators for certain limiting cases. For example, Calderon--Zygmund operators with standard kernels(such as the Hilbert transform on the real line and the Riesz transforms on $\mathbb R^n$) are known to be bounded on $L^p(\mathbb R^n)$ for all $1<p<\infty$, but not bounded on $L^p(\mathbb R^n)$ for $p=\infty$. Actually, in this case, the standard conditions on their kernels ensure that they are bounded from $L^\infty(\mathbb R^n)$ into $\mathrm{BMO}(\mathbb R^n)$. The Riesz potential operator of order $\alpha$ is known to be bounded from $L^p(\mathbb R^n)$ into $L^q(\mathbb R^n)$ when $1<p<q<\infty$ and $1/q=1/p-\alpha/n$, but not bounded from $L^p(\mathbb R^n)$ into $L^{\infty}(\mathbb R^n)$ when $p=n/{\alpha}$ and $0<\alpha<n$. In this case, it is actually bounded from $L^{n/{\alpha}}(\mathbb R^n)$ into $\mathrm{BMO}(\mathbb R^n)$ (see \cite{duoand}, \cite{Grafakos}, \cite{stein} and \cite{stein2}).
\end{rem}

For a measurable set $E\subseteq\mathbb R^n$ and for a measurable function $f$ on $E$, let us recall the \emph{distribution function} $d_{f;E}(\lambda)$ for $\lambda>0$, and \emph{non-increasing rearrangement} $f^{*}_{E}(t)$ for $t>0$. We define the distribution function of $f$ by
\begin{equation*}
d_{f;E}(\lambda):=m\big(\big\{x\in E:|f(x)|>\lambda\big\}\big),\quad \lambda>0.
\end{equation*}
It is not difficult to see that the distribution function $d_{f;E}(\lambda)$ is non-increasing and right continuous on $(0,\infty)$. Moreover, for $f\in L^p(E)$ with $1\leq p<\infty$, we obtain via Fubini's theorem that
\begin{equation}\label{cake1}
\|f\|_{L^p(E)}=\bigg(\int_0^{\infty}p\lambda^{p-1}d_{f;E}(\lambda)\,d\lambda\bigg)^{1/p}.
\end{equation}
The non-increasing rearrangement of $f$ is defined by
\begin{equation*}
f^{*}_{E}(t):=\inf\big\{\lambda>0:d_{f;E}(\lambda)\leq t\big\},\quad t>0.
\end{equation*}
This definition and the fact that $d_{f;E}$ is non-increasing imply that $f^{*}_{E}$ is non-increasing and right continuous too. Moreover,
\begin{equation*}
f^{*}_{E}(t)>\lambda\Longleftrightarrow d_{f;E}(\lambda)>t.
\end{equation*}
From this, it can be shown that the functions $f^{*}$ and $f$ have the same distribution function(equimeasurable functions), and hence
\begin{equation}\label{equality1}
\|f\|_{L^p(E)}=
\begin{cases}
\displaystyle\bigg(\int_{0}^{\infty}\big[f^{*}_{E}(t)\big]^p\,dt\bigg)^{1/p}, &\mbox{if}~ p\in[1,\infty),\\
\displaystyle\sup_{t>0}f^{*}_{E}(t)=\lim_{t\downarrow 0}f^{*}_{E}(t), &\mbox{if}~ p=\infty.
\end{cases}
\end{equation}
In particular, when $E=\mathbb R^n$, we write $d_{f;E}(\lambda)$ and $f^{*}_{E}(t)$ simply by $d_{f}(\lambda)$ and $f^{*}(t)$, respectively. Accordingly, for $f\in L^p(\mathbb R^n)$ with $1\leq p<\infty$, we have
\begin{equation}\label{cake2}
\|f\|_{L^p}=\bigg(\int_0^{\infty}p\lambda^{p-1}d_{f}(\lambda)\,d\lambda\bigg)^{1/p}.
\end{equation}
Moreover,
\begin{equation}\label{equality2}
\|f\|_{L^p}=
\begin{cases}
\displaystyle\bigg(\int_{0}^{\infty}\big[f^{*}(t)\big]^p\,dt\bigg)^{1/p}, &\mbox{if}~ p\in[1,\infty),\\
\displaystyle\sup_{t>0}f^{*}(t)=\lim_{t\downarrow 0}f^{*}(t), &\mbox{if}~ p=\infty.
\end{cases}
\end{equation}
By definition, we also see that
\begin{equation*}
d_{f;E}(\lambda)=d_{f\cdot\mathbf{1}_E}(\lambda),~\lambda>0, \quad \& \quad f^{*}_{E}(t)=(f\cdot1_{E})^{*}(t),~t>0.
\end{equation*}
Denote by $L^1_{\mathrm{loc}}(\mathbb R^n)$ the set of all locally integrable functions on $\mathbb R^n$. Let $1\leq p<\infty$. We first show that the space $L^p(\mathbb R^n)\cap L^\infty(\mathbb R^n)$ is continuously embedded into $L^q(\mathbb R^n)$ for all $q$ with $p<q<\infty$, and this embedding constant is equal to 1. In fact, if $f\in L^p(\mathbb R^n)\cap L^\infty(\mathbb R^n)$ with $1\leq p<\infty$, then for any $p<q<\infty$, by using \eqref{equality2}, we can deduce that
\begin{equation*}
\begin{split}
\|f\|_{L^q}&=\bigg(\int_0^{\infty}\big[f^{*}(t)\big]^qdt\bigg)^{1/q}\\
&=\bigg(\int_0^{\infty}\big[f^{*}(t)\big]^p\cdot\big[f^{*}(t)\big]^{q-p}dt\bigg)^{1/q}\\
&\leq\bigg(\int_0^{\infty}\big[f^{*}(t)\big]^pdt\bigg)^{1/q}\cdot\big(\|f\|_{L^{\infty}}\big)^{1-p/q}\\
&=\big(\|f\|_{L^p}\big)^{p/q}\cdot\big(\|f\|_{L^\infty}\big)^{1-p/q},
\end{split}
\end{equation*}
as desired. We denote by $f^{**}$ the average function of $f^{*}$(the action of the Hardy operator on $f^{*}$), which is defined by
\begin{equation*}
f^{**}(t):=\frac{\,1\,}{t}\int_0^{t}f^{*}(\tau)\,d\tau,\quad t>0.
\end{equation*}
Since $f^{*}$ is non-increasing, we see that for any $t>0$,
\begin{equation*}
\begin{split}
f^{**}(t)&=\frac{\,1\,}{t}\int_0^{t}f^{*}(\tau)\,d\tau\\
&\geq\frac{\,1\,}{t}\big[f^{*}(t)\cdot t\big]=f^{*}(t),
\end{split}
\end{equation*}
which implies that
\begin{equation}\label{geq}
f^{**}(t)-f^{*}(t)\geq 0,\quad t>0.
\end{equation}
For example: if $f^{*}(t)=\frac{1}{t+1}$, then $f^{**}(t)=\frac{\ln(1+t)}{t}$. It is easy to verify that $f^{**}(t)-f^{*}(t)\geq 0$.
This motivates Bennett, DeVore and Sharpley \cite{be1} to define a new function space $\mathcal{W}$ as follows
(see also Bennett and Sharpley \cite{be2}, Kozono and Wadade \cite{Kozono2}):
\begin{equation*}
\mathcal{W}:=\Big\{f\in L^1_{\mathrm{loc}}(\mathbb R^n):\|f\|_{\mathcal{W}}=\sup_{t>0}\big[f^{**}(t)-f^{*}(t)\big]<+\infty\Big\}.
\end{equation*}
In addition, for any $f\in L^p(\mathbb R^n)$ with $1\leq p<\infty$ and $t>0$, it follows from H\"{o}lder's inequality and \eqref{equality2} that
\begin{equation*}
\begin{split}
\int_0^{t}f^{*}(\tau)\,d\tau&\leq\bigg(\int_0^{t}\big[f^{*}(\tau)\big]^pd\tau\bigg)^{1/p}\cdot t^{1/{p'}}\\
&\leq\bigg(\int_0^{\infty}\big[f^{*}(\tau)\big]^pd\tau\bigg)^{1/p}\cdot t^{1/{p'}}=\|f\|_{L^p}\cdot t^{1/{p'}}.
\end{split}
\end{equation*}
Then we have
\begin{equation}\label{extend}
f^{**}(t)=\frac{\,1\,}{t}\int_0^{t}f^{*}(\tau)\,d\tau\leq\frac{1}{t^{1/p}}\cdot\|f\|_{L^p},\quad t>0.
\end{equation}
Motivated by the results in $L^p(\mathbb R^n)$ mentioned above, we consider the same problem in the setting of Lorentz spaces. We will show that $L^{p,r}(\mathbb R^n)\cap L^\infty(\mathbb R^n)$ is continuously embedded into $L^q(\mathbb R^n)$ for all $q$ with $p<q<\infty$, where $1\leq p<\infty$ and $p<r\leq\infty$, and this embedding constant is equal to 2 (see Theorem \ref{thmwang1} below). Meanwhile, if $f\in L^{p,r}(\mathbb R^n)$ with $1<p<\infty$ and $p\leq r\leq\infty$, then (see \eqref{key1} below)
\begin{equation*}
f^{**}(t)\leq\Big(\frac{p'}{r'}\Big)^{1/{r'}}
\frac{1}{t^{1/p}}\cdot\|f\|_{L^{p,r}},\quad t>0,
\end{equation*}
where $r'=r/{(r-1)}$ denotes the H\"{o}lder conjugate exponent of $r\in(1,\infty)$, and $r'=1$ when $r=\infty$. This is just \eqref{extend} with $1<p<\infty$ by putting $r=p$. Unfortunately, the above key estimate does not hold if $p=1$. By using the above estimate for $f^{**}$, we are able to show that the space $L^{p,r}(\mathbb R^n)\cap \mathcal{W}$ is continuously embedded into $L^q(\mathbb R^n)$ for all $q$ with $p<q<\infty$, where $1<p<\infty$ and $p\leq r\leq\infty$. As a consequence, we have $L^{p,r}(\mathbb R^n)\cap \mathrm{BMO}(\mathbb R^n)\hookrightarrow L^q(\mathbb R^n)$ since $\mathrm{BMO}(\mathbb R^n)$ is a subspace of $\mathcal{W}$ (see Theorem \ref{strong} and Corollary \ref{cor33} below).

On the other hand, let us now recall the definitions of the classical Morrey space and weak Morrey space. Let $\Omega$ be a connected open set of $\mathbb R^n$(bounded or unbounded). We denote by $\overline{\Omega}$ the closure of $\Omega$, and by $\mathrm{diam}\,\Omega$ the diameter of $\Omega$. For any $x_0\in\mathbb R^n$ and $r\in(0,\infty)$, set
\begin{equation*}
\Omega(x_0,r)=B(x_0,r)\cap\Omega,
\end{equation*}
where $B(x_0,r)$ denotes the open ball centered at $x_0$ with radius $r$, that is,
\begin{equation*}
B(x_0,r):=\big\{x\in\mathbb R^n:|x-x_0|<r\big\}.
\end{equation*}
For $1\leq p<\infty$ and $0\leq\kappa\leq1$, define
\begin{equation*}
\big\|f\big\|_{M^{p;\kappa}(\Omega)}:=\sup_{x_0\in\overline{\Omega},r\in(0,\mathrm{diam}\,\Omega)}
\bigg(\frac{1}{m(\Omega(x_0,r))^{\kappa}}\int_{\Omega(x_0,r)}|f(y)|^p\,dy\bigg)^{1/p}.
\end{equation*}
The Morrey space $M^{p;\kappa}(\Omega)$ is defined to be the set of all locally integrable functions $f$ on $\Omega$ for which $\|f\|_{M^{p;\kappa}(\Omega)}$ is finite. The above space was originally introduced by Morrey to investigate the local behaviour (regularity) of solutions to second order elliptic partial differential equations (see \cite{morrey}). In the case of $\Omega=\mathbb R^n$, $M^{p;\kappa}(\mathbb R^n)$ is the classical Morrey space, which is defined as the set of all locally integrable functions $f$ on $\mathbb R^n$ which satisfy
\begin{equation*}
\begin{split}
\|f\|_{M^{p;\kappa}}&:=\sup_{B\subset\mathbb R^n}\bigg(\frac{1}{m(B)^{\kappa}}\int_{B}|f(y)|^p\,dy\bigg)^{1/p}\\
&=\sup_{B\subset\mathbb R^n}\frac{1}{m(B)^{\kappa/p}}\|f\|_{L^{p}({B})}<+\infty,
\end{split}
\end{equation*}
where the supremum is taken over all balls $B$ in $\mathbb R^n$. Recall that the weak Morrey space $WM^{p;\kappa}(\mathbb R^n)$ is defined as the set of all measurable functions $f$ on $\mathbb R^n$ which satisfy
\begin{equation*}
\begin{split}
\|f\|_{WM^{p;\kappa}}&:=\sup_{B\subset\mathbb R^n}\sup_{\lambda>0}
\frac{1}{m(B)^{\kappa/p}}\lambda\cdot m\big(\big\{y\in B:|f(y)|>\lambda\big\}\big)^{1/p}\\
&=\sup_{B\subset\mathbb R^n}\frac{1}{m(B)^{\kappa/p}}\|f\|_{WL^{p}({B})}<+\infty,
\end{split}
\end{equation*}
where the supremum is taken over all balls $B\subset\mathbb R^n$ and all $\lambda>0$.
\begin{rem}
The space $M^{p;\kappa}(\mathbb R^n)$ becomes a Banach space with the norm $\|\cdot\|_{M^{p;\kappa}}$. Moreover, for $\kappa=0$ and $\kappa=1$, the Morrey spaces $M^{p;0}(\mathbb R^n)$ and $M^{p;1}(\mathbb R^n)$ coincide (with equality of norms) with the spaces $L^p(\mathbb R^n)$ and $L^{\infty}(\mathbb R^n)$, respectively. If $\kappa<0$ or $\kappa>1$, then the space $M^{p;\kappa}(\mathbb R^n)$ contains only the zero function.
It is clear that for $1\leq p<\infty$ and $0<\kappa<1$,
\begin{equation}\label{inlusion}
M^{p;\kappa}(\mathbb R^n)\subset WM^{p;\kappa}(\mathbb R^n),
\end{equation}
and we have that for any $f\in M^{p;\kappa}(\mathbb R^n)$,
\begin{equation*}
\|f\|_{WM^{p;\kappa}}\leq\|f\|_{M^{p;\kappa}},
\end{equation*}
since for any ball $B$ in $\mathbb R^n$, $\|f\|_{WL^p(B)}\leq \|f\|_{L^p(B)}$ as a direct consequence of Chebyshev's inequality. Moreover, it was proved in \cite{gun} that the weak Morrey space $WM^{p;\kappa_*}(\mathbb R^n)$ is continuously embedded into the Morrey space $M^{q;\kappa}(\mathbb R^n)$ when $1\leq q<p<\infty$, $0<\kappa_*<1$ and $(1-\kappa)p=(1-\kappa_*)q$, and the size of the embedding constant from weak Morrey spaces into Morrey spaces is specified, see \cite[Theorem 1.1]{gun}.
\end{rem}

\begin{rem}
Let $0\leq\kappa<1$. It is well known that the boundedness of the Hardy--Littlewood maximal operator $M$ on Banach function spaces plays a key role in harmonic analysis. Within the framework of Morrey spaces, we know that the Hardy--Littlewood maximal operator $M$ is bounded from $M^{1;\kappa}(\mathbb R^n)$ into $WM^{1;\kappa}(\mathbb R^n)$ by \cite{chia} (see also \cite[Theorem 3.2]{komori} and \cite[Theorem 1]{nakai2}), but it is not bounded on $M^{1;\kappa}(\mathbb R^n)$ by \cite{nakai} (see also \cite[Proposition 6.1]{lida} and \cite[Lemma 3.5]{sawano}). This suggests that there exists a function $\mathcal{F}\in WM^{1;\kappa}(\mathbb R^n)$ such that $\mathcal{F}\notin M^{1;\kappa}(\mathbb R^n)$. A constructive proof of this fact can be found in \cite{gun}. Moreover, it can be shown that there exists a function $\mathcal{F}\in WM^{p;\kappa}(\mathbb R^n)$ such that $\mathcal{F}\notin M^{p;\kappa}(\mathbb R^n)$ for $1<p<\infty$ and $0<\kappa<1$ (see, for instance, \cite[Theorem 1.2]{gun}). That is to say, the inclusion \eqref{inlusion} is strict.
\end{rem}
For more information on the theory of Morrey spaces, see \cite{adams}, \cite{adams1}, \cite{adamsxiao} and \cite{peetre}.

Let $1\leq p<\infty$ and $0<\kappa<1$. Let us now prove that the space $M^{p;\kappa}(\mathbb R^n)\cap L^\infty(\mathbb R^n)$ is continuously embedded into $M^{q;\kappa}(\mathbb R^n)$ for all $q$ with $p<q<\infty$, and this embedding constant is equal to 1. In fact, if $f\in M^{p;\kappa}(\mathbb R^n)\cap L^\infty(\mathbb R^n)$, then for any $p<q<\infty$ and for any ball $\mathcal{B}$ in $\mathbb R^n$, by using \eqref{equality1}, we have
\begin{equation*}
\begin{split}
\frac{1}{m(\mathcal{B})^{\kappa/q}}\|f\|_{L^{q}(\mathcal{B})}
&=\frac{1}{m(\mathcal{B})^{\kappa/q}}\bigg(\int_0^{\infty}\big[f^{*}_{\mathcal{B}}(t)\big]^qdt\bigg)^{1/q}\\
&=\frac{1}{m(\mathcal{B})^{\kappa/q}}
\bigg(\int_0^{\infty}\big[f^{*}_{\mathcal{B}}(t)\big]^p\cdot\big[f^{*}_{\mathcal{B}}(t)\big]^{q-p}dt\bigg)^{1/q}.\\
\end{split}
\end{equation*}
It is clear that for any fixed ball $\mathcal{B}\subset\mathbb R^n$ and $\lambda>0$,
\begin{equation*}
d_{f;\mathcal{B}}(\lambda)\leq d_{f}(\lambda),
\end{equation*}
which in turn implies that for all $t>0$,
\begin{equation*}
f^{*}_{\mathcal{B}}(t)\leq f^{*}(t)\leq \|f\|_{L^{\infty}}=\sup_{t>0}f^{*}(t).
\end{equation*}
Therefore, by using \eqref{equality1} again, we further obtain
\begin{equation*}
\begin{split}
\frac{1}{m(\mathcal{B})^{\kappa/q}}\|f\|_{L^{q}(\mathcal{B})}
&\leq\bigg(\frac{1}{m(\mathcal{B})^{\kappa}}\int_0^{\infty}\big[f^{*}_{\mathcal{B}}(t)\big]^pdt\bigg)^{1/q}
\cdot\big(\|f\|_{L^{\infty}}\big)^{1-p/q}\\
&\leq\big(\|f\|_{M^{p;\kappa}}\big)^{p/q}\cdot\big(\|f\|_{L^\infty}\big)^{1-p/q}.
\end{split}
\end{equation*}
Taking the supremum over all balls $\mathcal{B}$ in $\mathbb R^n$, we conclude that
\begin{equation*}
\|f\|_{M^{q;\kappa}}\leq\big(\|f\|_{M^{p;\kappa}}\big)^{p/q}\cdot\big(\|f\|_{L^\infty}\big)^{1-p/q}.
\end{equation*}
This is our desired estimate.

The purpose of this paper is to define a class of new Lorentz--Morrey spaces $LM^{p,r;\kappa}(\mathbb R^n)$. Motivated by the above result in $M^{p;\kappa}(\mathbb R^n)$,
we consider the same problem in these Lorentz--Morrey spaces, and prove that $LM^{p,r;\kappa}(\mathbb R^n)\cap L^\infty(\mathbb R^n)$ is continuously embedded into $M^{q;\kappa}(\mathbb R^n)$ for all $q$ with $p<q<\infty$, where $1\leq p<\infty$, $p<r\leq\infty$ and $0<\kappa<1$. Moreover, this embedding constant is equal to 2 (see Theorem \ref{thmwang2} below). Then we are going to prove that
\begin{equation*}
LM^{p,r;\kappa}(\mathbb R^n)\cap \mathrm{BMO}(\mathbb R^n)\hookrightarrow M^{q;\kappa}(\mathbb R^n)
\end{equation*}
for all $q$ with $p<q<\infty$, where $1<p<\infty$, $p\leq r\leq\infty$ and $0<\kappa<1$ (see Theorem \ref{weak} below).

\section{Notations and definitions}

Let us recall the definition of Lorentz spaces and present their basic properties. See \cite[Chapter 1]{Grafakos2} and \cite[Chapter 5]{sw} for more details.
\begin{defn}[\cite{Grafakos2,sw}]
Let $1\leq p<\infty$ and $1\leq r\leq\infty$. Then the Lorentz space $L^{p,r}(\mathbb R^n)$ with indices $p$ and $r$ is defined to be the set of all measurable functions $f$ on $\mathbb R^n$ such that $\|f\|_{L^{p,r}}<+\infty$, where the functional $\|\cdot\|_{L^{p,r}}$ is defined by
\begin{equation*}
\|f\|_{L^{p,r}}:=
\begin{cases}
\displaystyle\bigg(\int_0^{\infty}\Big[s^{1/p}f^{*}(s)\Big]^r\frac{ds}{s}\bigg)^{1/r}, &\mbox{if}~ 1\leq r<\infty,\\
\displaystyle\sup_{s>0}\Big[s^{1/p}f^{*}(s)\Big], &\mbox{if}~ r=\infty.
\end{cases}
\end{equation*}
\end{defn}
The index $p$ in $L^{p,r}(\mathbb R^n)$ is called the \emph{principal index}, and $q$ is called the \emph{secondary index}. For $1\leq p<\infty$ and $r=p$, by using \eqref{equality2}, we have
\begin{equation}\label{lorentz1}
\|f\|_{L^{p,p}}=\bigg(\int_0^{\infty}\Big[s^{1/p}f^{*}(s)\Big]^p\frac{ds}{s}\bigg)^{1/p}
=\bigg(\int_{0}^{\infty}\big[f^{*}(s)\big]^pds\bigg)^{1/p}=\|f\|_{L^p},
\end{equation}
and hence $L^{p,p}(\mathbb R^n)=L^{p}(\mathbb R^n)$. Moreover, for $1\leq p<\infty$ and $r=\infty$, we can show that $\|f\|_{L^{p,\infty}}=\|f\|_{WL^{p}}$, which is based on the following result. Therefore, $L^{p,\infty}(\mathbb R^n)=WL^{p}(\mathbb R^n)$.

\begin{prop}[\cite{Grafakos2,sw}]\label{prop21}
Let $f$ be a measurable function on the set $E\subseteq\mathbb R^n$ and $1\leq p<\infty$. Then we have the identity
\begin{equation*}
\sup_{s>0}s^{1/p}f^{*}_{E}(s)=\sup_{\lambda>0}\lambda d_{f;E}(\lambda)^{1/p}.
\end{equation*}
In particular, when $E=\mathbb R^n$, we have the identity
\begin{equation*}
\sup_{s>0}s^{1/p}f^{*}(s)=\sup_{\lambda>0}\lambda d_{f}(\lambda)^{1/p}.
\end{equation*}
\end{prop}
Note that $L^{p}(\mathbb R^n)=L^{p,p}(\mathbb R^n)\subset L^{p,\infty}(\mathbb R^n)$ with $1\leq p<\infty$. The next result is well known in harmonic analysis, which shows that for any fixed $p$, the Lorentz spaces $L^{p,r}(\mathbb R^n)$ are increasing with respect to $r$.
\begin{lem}[\cite{Grafakos2,sw}]\label{prop22}
Let $1\leq p<\infty$ and $1\leq q<r\leq\infty$. Then $L^{p,q}(\mathbb R^n)\subset L^{p,r}(\mathbb R^n)$, and
\begin{equation*}
\|f\|_{L^{p,r}}\lesssim\|f\|_{L^{p,q}}
\end{equation*}
with the implicit positive constant independent of $f$. In particular, we have
\begin{equation*}
L^{p}(\mathbb R^n)=L^{p,p}(\mathbb R^n)\subset L^{p,q}(\mathbb R^n)\subset L^{p,\infty}(\mathbb R^n)
\end{equation*}
for any $p<q<\infty$.
\end{lem}

The following important formula tells us that we can compute $\|\cdot\|_{L^{p,r}}$ in terms of the distribution function, which may be regarded as a generalization of \eqref{cake1} and \eqref{cake2}.
\begin{prop}[\cite{Grafakos2,sw}]\label{prop23}
Let $1\leq p<\infty$ and $1\leq r<\infty$. For given $f\in L^{p,r}(\mathbb R^n)$, we have the identity
\begin{equation*}
\|f\|_{L^{p,r}}=\bigg(p\int_0^{\infty}\Big[\lambda d_f(\lambda)^{1/p}\Big]^r\frac{d\lambda}{\lambda}\bigg)^{1/r}<+\infty.
\end{equation*}
\end{prop}

We are now in a position to state the definition of our new spaces.
\begin{defn}
Let $1\leq p<\infty$, $1\leq r\leq\infty$ and $0<\kappa<1$. Then the Lorentz--Morrey space $LM^{p,r;\kappa}(\mathbb R^n)$ with indices $p,r$ and $\kappa$ is defined as the set of all measurable functions $f$ on $\mathbb R^n$ which satisfy
\begin{equation*}
\|f\|_{LM^{p,r;\kappa}}:=\sup_{B\subset\mathbb R^n}\frac{1}{m(B)^{\kappa/p}}\|f\|_{L^{p,r}(B)}<+\infty,
\end{equation*}
where the supremum is taken over all balls $B$ in $\mathbb R^n$ and we define
\begin{equation*}
\|f\|_{L^{p,r}(B)}=\|f\cdot \mathbf{1}_{B}\|_{L^{p,r}}:=
\begin{cases}
\displaystyle\bigg(\int_0^{\infty}\Big[s^{1/p}f^{*}_{B}(s)\Big]^r\frac{ds}{s}\bigg)^{1/r}, &\mbox{if}~ 1\leq r<\infty,\\
\displaystyle\sup_{s>0}\Big[s^{1/p}f^{*}_{B}(s)\Big], &\mbox{if}~ r=\infty.
\end{cases}
\end{equation*}
\end{defn}
We state and prove some basic properties of such spaces. First of all, for $1\leq p<\infty$, $0<\kappa<1$ and $r=p$, it follows directly from \eqref{lorentz1} that
\begin{equation*}
\begin{split}
\|f\|_{LM^{p,p;\kappa}}&=\sup_{B\subset\mathbb R^n}\frac{1}{m(B)^{\kappa/p}}\|f\cdot \mathbf{1}_{B}\|_{L^{p}}\\
&=\sup_{B\subset\mathbb R^n}\frac{1}{m(B)^{\kappa/p}}\|f\|_{L^{p}(B)}=\|f\|_{M^{p;\kappa}},
\end{split}
\end{equation*}
and hence $LM^{p,p;\kappa}(\mathbb R^n)=M^{p;\kappa}(\mathbb R^n)$. Moreover, for $1\leq p<\infty$, $0<\kappa<1$ and $r=\infty$, it follows from Proposition \ref{prop21} that
\begin{equation*}
\begin{split}
\|f\|_{LM^{p,\infty;\kappa}}&=\sup_{B\subset\mathbb R^n}\frac{1}{m(B)^{\kappa/p}}\sup_{\lambda>0}\lambda d_{f;B}(\lambda)^{1/p}\\
&=\sup_{B\subset\mathbb R^n}\frac{1}{m(B)^{\kappa/p}}\|f\|_{WL^{p}(B)}=\|f\|_{WM^{p;\kappa}}.
\end{split}
\end{equation*}
Therefore, we have $LM^{p,\infty;\kappa}(\mathbb R^n)=WM^{p;\kappa}(\mathbb R^n)$. Since for each fixed ball $B$ in $\mathbb R^n$ and for $1\leq p<\infty$, in view of Lemma \ref{prop22},
\begin{equation*}
\|f\cdot \mathbf{1}_{B}\|_{L^{p,r}}\lesssim\|f\cdot \mathbf{1}_{B}\|_{L^{p,q}}
\end{equation*}
holds whenever $1\leq q<r\leq\infty$. As a consequence, we have the following result, which tells us that for some fixed $p$ and $\kappa$, the Lorentz--Morrey spaces $LM^{p,r;\kappa}(\mathbb R^n)$ are increasing with respect to the second index $r$.
\begin{lem}
Let $1\leq p<\infty$, $1\leq q<r\leq\infty$ and $0<\kappa<1$. Then $LM^{p,q;\kappa}(\mathbb R^n)\subset LM^{p,r;\kappa}(\mathbb R^n)$, and
\begin{equation*}
\|f\|_{LM^{p,r;\kappa}}\lesssim\|f\|_{LM^{p,q;\kappa}}
\end{equation*}
with the implicit positive constant independent of $f$. In particular, we have
\begin{equation*}
M^{p;\kappa}(\mathbb R^n)=LM^{p,p;\kappa}(\mathbb R^n)\subset LM^{p,q;\kappa}(\mathbb R^n)\subset LM^{p,\infty;\kappa}(\mathbb R^n)
\end{equation*}
for any $p<q<\infty$.
\end{lem}

Next, based on Proposition \ref{prop23} and the fact that $f^{*}_{B}(s)=(f\cdot1_{B})^{*}(s)$, we can also compute $\|\cdot\|_{LM^{p,r;\kappa}}$ in terms of the distribution function.
\begin{prop}\label{prop24}
Let $1\leq p<\infty$, $1\leq r<\infty$ and $0<\kappa<1$. For given $f\in LM^{p,r;\kappa}(\mathbb R^n)$, we have the identity
\begin{equation*}
\|f\|_{LM^{p,r;\kappa}}=\sup_{B\subset\mathbb R^n}\frac{1}{m(B)^{\kappa/p}}
\bigg(p\int_0^{\infty}\Big[\lambda d_{f;B}(\lambda)^{1/p}\Big]^r\frac{d\lambda}{\lambda}\bigg)^{1/r}<+\infty.
\end{equation*}
\end{prop}

This identity, together with estimates of the distribution function, leads to the following inclusion relation, which is a substantial improvement and extension of known results in \cite[Theorem 1.1]{gun}.
\begin{thm}\label{appendix}
Let $1\leq q<p<\infty$ and $0<\kappa_{*}<1$. Then for $0<\kappa<1$ and $q\leq r<\infty$, we have
\begin{equation*}
LM^{p,\infty;\kappa_{*}}(\mathbb R^n)\subset LM^{q,r;\kappa}(\mathbb R^n)
\end{equation*}
whenever $(1-\kappa)p=(1-\kappa_*)q$, and
\begin{equation*}
\|f\|_{LM^{q,r;\kappa}}\leq \Big[\frac{pq}{r(p-q)}\Big]^{\frac{1}{r}}\|f\|_{LM^{p,\infty;\kappa_{*}}}.
\end{equation*}
In particular, when $r=q$,
\begin{equation*}
WM^{p;\kappa_{*}}(\mathbb R^n)\subset M^{q;\kappa}(\mathbb R^n)
\end{equation*}
whenever $(1-\kappa)p=(1-\kappa_{*})q$, and
\begin{equation*}
\|f\|_{M^{q;\kappa}}\leq \Big[\frac{p}{p-q}\Big]^{\frac{1}{q}}\|f\|_{WM^{p;\kappa_{*}}}.
\end{equation*}
\end{thm}
The proof of Theorem \ref{appendix} will be given in the Appendix.
\begin{rem}
In general, for $1<p<\infty$ and $1\leq r\leq\infty$, the functional $\|\cdot\|_{LM^{p,r;\kappa}}$ is not a norm unless $r=p$, since it does not satisfy the triangle inequality. Actually, it is a quasi-norm since
\begin{equation*}
\begin{split}
\|(f+g)\cdot \mathbf{1}_{B}\|_{L^{p,r}}&=\|f\cdot \mathbf{1}_{B}+g\cdot \mathbf{1}_{B}\|_{L^{p,r}}\\
&\leq 2\big(\|f\cdot \mathbf{1}_{B}\|_{L^{p,r}}+\|g\cdot \mathbf{1}_{B}\|_{L^{p,r}}\big),
\end{split}
\end{equation*}
it is easy to verify that for any given functions $f,g\in LM^{p,r;\kappa}(\mathbb R^n)$,
\begin{equation*}
\|f+g\|_{LM^{p,r;\kappa}}\leq 2\big(\|f\|_{LM^{p,r;\kappa}}+\|g\|_{LM^{p,r;\kappa}}\big).
\end{equation*}
However, when $1<p<\infty$ and $1\leq r\leq\infty$, if we replace $f^{*}$ in the definition of $\|f\|_{LM^{p,r;\kappa}}$ with $f^{**}$, we get a quantity which is equivalent to $\|f\|_{LM^{p,r;\kappa}}$. To be more precise, define the quantity $\|f\|^{\ast}_{LM^{p,r;\kappa}}$ by
\begin{equation*}
\|f\|^{\ast}_{LM^{p,r;\kappa}}:=\sup_{B\subset\mathbb R^n}\frac{1}{m(B)^{\kappa/p}}\|f\|^{\ast}_{L^{p,r}(B)},
\end{equation*}
where the supremum is taken over all balls $B$ in $\mathbb R^n$ and
\begin{equation*}
\|f\|^{\ast}_{L^{p,r}(B)}:=
\begin{cases}
\displaystyle\bigg(\int_0^{\infty}\Big[s^{1/p}f^{**}_{B}(s)\Big]^r\frac{ds}{s}\bigg)^{1/r}, &\mbox{if}~ 1\leq r<\infty,\\
\displaystyle\sup_{s>0}\Big[s^{1/p}f^{**}_{B}(s)\Big], &\mbox{if}~ r=\infty.
\end{cases}
\end{equation*}
We can show that the functional $\|\cdot\|^{\ast}_{LM^{p,r;\kappa}}$ satisfies the triangle inequality, where $1<p<\infty$, $1\leq r\leq\infty$ and $0<\kappa<1$, and hence it is indeed a norm. In this situation, $LM^{p,r;\kappa}(\mathbb R^n)$ becomes a Banach space under the norm $\|\cdot\|^{\ast}_{LM^{p,r;\kappa}}$. Moreover,
\begin{equation*}
\|f\|_{LM^{p,r;\kappa}}\leq\|f\|^{\ast}_{LM^{p,r;\kappa}}\leq\frac{p}{p-1}\|f\|_{LM^{p,r;\kappa}},
\end{equation*}
when $1<p<\infty$, $1\leq r\leq\infty$ and $0<\kappa<1$. The proof is based on the corresponding estimate on the Lorentz spaces.
\end{rem}

Finally, we present the following three results needed for the proofs of our main theorems.
\begin{lem}[\cite{Kozono2}]\label{lemmaw}
There exists a positive constant $C(n)$ depending only on the dimension $n$ such that
\begin{equation*}
\|f\|_{\mathcal{W}}\leq C(n)\|f\|_{\mathrm{BMO}}
\end{equation*}
holds for all $f\in \mathrm{BMO}(\mathbb R^n)$.
\end{lem}

Recall that for $\alpha>0$, the gamma function is defined by
\begin{equation*}
\Gamma(\alpha):=\int_0^{\infty}\mu^{\alpha-1} e^{-\mu}\,d\mu.
\end{equation*}
For all $t>0$, a change of variables implies that if $\mu=\log(t/s)$, then
\begin{equation}\label{weneed}
\bigg(\int_0^{t}\Big[\log\frac{t}{\,s\,}\Big]^qds\bigg)^{1/q}
=\bigg(\int_0^{\infty}t\cdot\mu^q e^{-\mu}\,d\mu\bigg)^{1/q}=t^{1/q}\cdot\Gamma(q+1)^{1/q},
\end{equation}
when $1\leq q<\infty$. In addition, we have the following elementary inequality, which says that for any two positive numbers $A,B$ and $1\leq q<\infty$,
\begin{equation}\label{elemen}
A^{1/q}+B^{1/q}\leq 2^{1-1/q}\cdot(A+B)^{1/q}.
\end{equation}

\section{Interpolation inequalities between Lorentz, $L^{\infty}$ and BMO spaces}
\subsection{Main results}
The main results of this section are stated as follows.

\begin{thm}\label{thmwang1}
Let $n\in \mathbb{N}$, $1\leq p<\infty$ and $p<r\leq\infty$. Then the following statements hold.

$(1)$ If $r<\infty$ and $f\in L^{p,r}(\mathbb R^n)\cap{L^\infty}(\mathbb R^n)$, then for every $q$ with $r<q<\infty$, we have $f\in L^q(\mathbb R^n)$, and the interpolation inequality
\begin{equation*}
\|f\|_{L^q}\leq 2\cdot\big(\|f\|_{L^{p,r}}\big)^{p/q}\cdot\big(\|f\|_{L^\infty}\big)^{1-p/q}
\end{equation*}
holds for all $f\in L^{p,r}(\mathbb R^n)\cap{L^\infty}(\mathbb R^n)$.

$(2)$ If $r=\infty$ and $f\in L^{p,r}(\mathbb R^n)\cap{L^\infty}(\mathbb R^n)$, then for every $q$ with $p<q<\infty$, we have $f\in L^q(\mathbb R^n)$, and the interpolation inequality
\begin{equation*}
\|f\|_{L^q}\leq \Big[2^{1-1/q}\Big(\frac{q}{q-p}\Big)^{1/q}\Big]
\cdot\big(\|f\|_{L^{p,\infty}}\big)^{p/q}\cdot\big(\|f\|_{L^\infty}\big)^{1-p/q}
\end{equation*}
holds for all $f\in L^{p,\infty}(\mathbb R^n)\cap{L^\infty}(\mathbb R^n)$.
\end{thm}

\begin{thm}\label{strong}
Let $n\in \mathbb{N}$, $1<p<\infty$ and $p\leq r\leq\infty$. If $p\leq r<\infty$ and $f\in L^{p,r}(\mathbb R^n)\cap \mathcal{W}$, then for every $q$ with $r<q<\infty$(or if $f\in L^{p,\infty}(\mathbb R^n)\cap \mathcal{W}$, then for every $q$ with $p<q<\infty$), we have $f\in L^q(\mathbb R^n)$, and the following interpolation inequality
\begin{equation*}
\|f\|_{L^q}\leq C\cdot q\big(\|f\|_{L^{p,r}}\big)^{p/q}\cdot\big(\|f\|_{\mathcal{W}}\big)^{1-p/q}
\end{equation*}
holds for all $f\in L^{p,r}(\mathbb R^n)\cap \mathcal{W}$. Here $C$ is an absolute positive constant independent of $q$.
\end{thm}

As an immediate consequence of Theorem $\ref{strong}$ and Lemma \ref{lemmaw}, we have the following result.
\begin{cor}\label{cor33}
Let $n\in \mathbb{N}$, $1<p<\infty$ and $p\leq r\leq\infty$. If $p\leq r<\infty$ and $f\in L^{p,r}(\mathbb R^n)\cap \mathrm{BMO}(\mathbb R^n)$, then for every $q$ with $r<q<\infty$(or if $f\in L^{p,\infty}(\mathbb R^n)\cap \mathrm{BMO}(\mathbb R^n)$, then for every $q$ with $p<q<\infty$), we have $f\in L^q(\mathbb R^n)$, and the following interpolation inequality
\begin{equation*}
\|f\|_{L^q}\leq C\cdot q\big(\|f\|_{L^{p,r}}\big)^{p/q}\cdot\big(\|f\|_{\mathrm{BMO}}\big)^{1-p/q}
\end{equation*}
holds for all $f\in L^{p,r}(\mathbb R^n)\cap \mathrm{BMO}(\mathbb R^n)$ with an absolute positive constant $C$ independent of $q$.
\end{cor}

\subsection{Proofs of Theorems $\ref{thmwang1}$ and $\ref{strong}$}
\begin{proof}[Proof of Theorem $\ref{thmwang1}$]
Let $f\in L^{p,r}(\mathbb R^n)\cap{L^\infty}(\mathbb R^n)$ with $1\leq p<\infty$ and $p<r\leq\infty$. For any $p<q<\infty$, in view of \eqref{equality2}, we have
\begin{equation*}
\begin{split}
\|f\|_{L^q}&=\bigg(\int_0^{\infty}\big[f^{*}(s)\big]^qds\bigg)^{1/q}\\
&\leq\bigg(\int_0^{t}\big[f^{*}(s)\big]^qds\bigg)^{1/q}+\bigg(\int_{t}^{\infty}\big[f^{*}(s)\big]^qds\bigg)^{1/q}\\
&:=I_1(t)+I_2(t),
\end{split}
\end{equation*}
where $t$ is a positive constant to be determined later. Recall that for any $s>0$,
\begin{equation}\label{often1}
f^{*}(s)\leq \sup_{s>0}f^{*}(s)=\|f\|_{L^{\infty}}.
\end{equation}
Consequently,
\begin{equation*}
I_1(t)\leq\|f\|_{L^\infty}\cdot t^{1/q}.
\end{equation*}
(1) When $r<\infty$ and $p<r<q<\infty$, in this case, it follows directly from \eqref{often1} that
\begin{equation*}
\begin{split}
I_2(t)&=\bigg(\int_{t}^{\infty}s^{r/p-1}\big[f^{*}(s)\big]^{r}\cdot\big[f^{*}(s)\big]^{q-r}\frac{1}{s^{r/p-1}}ds\bigg)^{1/q}\\
&\leq\bigg(\int_{t}^{\infty}\Big[s^{1/p}f^{*}(s)\Big]^r\frac{ds}{s}\bigg)^{1/q}
\cdot\big(\|f\|_{L^\infty}\big)^{1-r/q}\Big(\frac{1}{t^{r/p-1}}\Big)^{1/q},
\end{split}
\end{equation*}
where in the last inequality we have used the facts that $q-r>0$ and $r/p-1>0$. Furthermore, by the definition of $L^{p,r}(\mathbb R^n)$, we obtain
\begin{equation*}
\begin{split}
I_2(t)&\leq\bigg(\int_{0}^{\infty}\Big[s^{1/p}f^{*}(s)\Big]^r\frac{ds}{s}\bigg)^{1/q}
\cdot\big(\|f\|_{L^\infty}\big)^{1-r/q}\Big(\frac{1}{t^{r/p-1}}\Big)^{1/q}\\
&=\big(\|f\|_{L^{p,r}}\big)^{r/q}\big(\|f\|_{L^\infty}\big)^{1-r/q}\cdot\Big(\frac{1}{t^{r/p-1}}\Big)^{1/q}.
\end{split}
\end{equation*}
We now choose $t>0$ such that
\begin{equation*}
\big(\|f\|_{L^{p,r}}\big)^{r/q}\big(\|f\|_{L^\infty}\big)^{1-r/q}\cdot\Big(\frac{1}{t^{r/p-1}}\Big)^{1/q}
=\|f\|_{L^\infty}\cdot t^{1/q},
\end{equation*}
that is,
\begin{equation*}
t:=\left(\frac{\|f\|_{L^{p,r}}}{\|f\|_{L^\infty}}\right)^p.
\end{equation*}
Therefore,
\begin{equation*}
\|f\|_{L^q}\leq 2\cdot\big(\|f\|_{L^{p,r}}\big)^{p/q}\cdot\big(\|f\|_{L^\infty}\big)^{1-p/q}.
\end{equation*}
(2) When $r=\infty$ and $p<q<\infty$, in this case, it follows from the definition of $L^{p,\infty}(\mathbb R^n)$ and the fact $q/p>1$ that
\begin{equation*}
\begin{split}
I_2(t)&=\bigg(\int_{t}^{\infty}\big[s^{1/p}f^{*}(s)\big]^{q}\frac{1}{s^{q/p}}ds\bigg)^{1/q}\\
&\leq\Big[\sup_{s>0}s^{1/p}f^{*}(s)\Big]\cdot\bigg(\int_{t}^{\infty}\frac{1}{s^{q/p}}ds\bigg)^{1/q}\\
&=\|f\|_{L^{p,\infty}}\cdot\Big(\frac{1}{q/p-1}\Big)^{1/q} t^{1/q-1/p}.
\end{split}
\end{equation*}
We may choose $t>0$ such that
\begin{equation*}
\|f\|_{L^{p,\infty}}\cdot t^{1/q-1/p}=\|f\|_{L^\infty}\cdot t^{1/q},
\end{equation*}
that is,
\begin{equation*}
t:=\left(\frac{\|f\|_{L^{p,\infty}}}{\|f\|_{L^\infty}}\right)^p.
\end{equation*}
Therefore, in view of \eqref{elemen}, we conclude that
\begin{equation*}
\begin{split}
\|f\|_{L^q}&\leq\Big[\Big(\frac{p}{q-p}\Big)^{1/q}+1\Big]
\cdot\big(\|f\|_{L^{p,\infty}}\big)^{p/q}\cdot\big(\|f\|_{L^\infty}\big)^{1-p/q}\\
&\leq \Big[2^{1-1/q}\Big(\frac{q}{q-p}\Big)^{1/q}\Big]
\cdot\big(\|f\|_{L^{p,\infty}}\big)^{p/q}\cdot\big(\|f\|_{L^\infty}\big)^{1-p/q}.
\end{split}
\end{equation*}
This finishes the proof of Theorem $\ref{thmwang1}$.
\end{proof}

\begin{rem}
For $1\leq p<\infty$, it is easy to verify that
\begin{equation}\label{22}
\lim_{q\to\infty}2^{1-1/q}\Big(\frac{q}{q-p}\Big)^{1/q}=2,
\end{equation}
that is to say $2^{1-1/q}\big(\frac{q}{q-p}\big)^{1/q}=\mathcal{O}(2)$, as $q\to\infty$. Theorem $\ref{thmwang1}$ tells us that the embedding constant for the case $p<r<\infty$ is 2 and the asymptotically embedding constant for the case $p<r=\infty$ is also 2.
\end{rem}

\begin{proof}[Proof of Theorem $\ref{strong}$]
Let $f\in L^{p,r}(\mathbb R^n)\cap \mathcal{W}$ with $1<p<\infty$ and $p\leq r\leq\infty$, and let $p'$ be the exponent conjugate to $p$. First we claim that for any $t>0$,
\begin{equation}\label{key1}
f^{**}(t)\leq\Big(\frac{p'}{r'}\Big)^{1/{r'}}
\frac{1}{t^{1/p}}\cdot\|f\|_{L^{p,r}}.
\end{equation}
Here $r'=1$ when $r=\infty$.
In fact, we first suppose $p\leq r<\infty$. For any $t>0$, by using H\"{o}lder's inequality with exponent $r$, we obtain
\begin{equation*}
\begin{split}
\int_0^{t}f^{*}(\tau)\,d\tau&=\int_0^{t}\tau^{1/p-1/r}f^{*}(\tau)\cdot\tau^{1/r-1/p}\,d\tau\\
&\leq\bigg(\int_0^{t}\Big[\tau^{1/p}f^{*}(\tau)\Big]^r\frac{d\tau}{\tau}\bigg)^{1/r}
\times\bigg(\int_0^{t}\tau^{(1/r-1/p)r'}d\tau\bigg)^{1/{r'}}\\
&\leq \|f\|_{L^{p,r}}\bigg(\int_0^{t}\tau^{(1/r-1/p)r'}d\tau\bigg)^{1/{r'}}.
\end{split}
\end{equation*}
Recall that $p>1$. A simple computation leads to that
\begin{equation*}
(1/r-1/p)r'+1>0.
\end{equation*}
Then we have
\begin{equation*}
\begin{split}
\bigg(\int_0^{t}\tau^{(1/r-1/p)r'}\,d\tau\bigg)^{1/{r'}}
&=\Big[\frac{1}{(1/r-1/p)r'+1}\cdot t^{(1/r-1/p)r'+1}\Big]^{1/{r'}}\\
&=\Big(\frac{1}{r'}\Big)^{1/{r'}}\Big(\frac{1}{1-1/p}\Big)^{1/{r'}}\cdot t^{1-1/p}.
\end{split}
\end{equation*}
Hence, for any $t>0$,
\begin{equation*}
\begin{split}
f^{**}(t)&=\frac{\,1\,}{t}\int_0^{t}f^{*}(\tau)\,d\tau\\
&\leq\frac{\,1\,}{t}\Big(\frac{1}{r'}\Big)^{1/{r'}}\big(p'\big)^{1/{r'}}
\cdot t^{1-1/p}\|f\|_{L^{p,r}}\\
&=\Big(\frac{p'}{r'}\Big)^{1/{r'}}\frac{1}{t^{1/p}}\cdot\|f\|_{L^{p,r}},
\end{split}
\end{equation*}
as desired. Now we suppose $r=\infty$. In this case, it is easy to verify that
\begin{equation*}
\begin{split}
\int_0^{t}f^{*}(\tau)\,d\tau&=\int_0^{t}\tau^{1/p}f^{*}(\tau)\cdot\tau^{-1/p}\,{d\tau}\\
&\leq\Big[\sup_{\tau>0}\tau^{1/p}f^{*}(\tau)\Big]\cdot\int_0^{t}\tau^{-1/p}\,{d\tau}\\
&=\|f\|_{L^{p,\infty}}\frac{p}{p-1}\cdot t^{1-1/p},
\end{split}
\end{equation*}
where in the last step we have used the fact that $p>1$. Therefore, for any $t>0$,
\begin{equation*}
\begin{split}
f^{**}(t)&=\frac{\,1\,}{t}\int_0^{t}f^{*}(\tau)\,d\tau\\
&\leq p'\frac{1}{t^{1/p}}\cdot\|f\|_{L^{p,\infty}}.
\end{split}
\end{equation*}
This proves \eqref{key1}. It should be pointed out that inequality \eqref{key1} fails when $p=1$. On the other hand, since
\begin{equation}\label{important}
\begin{split}
\frac{d}{d\tau}\big(f^{**}(\tau)\big)&=-\frac{1}{\tau^2}\int_0^{\tau}f^{*}(\eta)\,d\eta+\frac{f^{*}(\tau)}{\tau}\\
&=\frac{f^{*}(\tau)-1/{\tau}\int_0^{\tau}f^{*}(\eta)\,d\eta}{\tau}\\
&=\frac{f^{*}(\tau)-f^{**}(\tau)}{\tau},
\end{split}
\end{equation}
so we have
\begin{equation*}
-\frac{d}{d\tau}\big(f^{**}(\tau)\big)=\frac{f^{**}(\tau)-f^{*}(\tau)}{\tau}\leq\frac{\|f\|_{\mathcal{W}}}{\tau},
\end{equation*}
for given $f\in \mathcal{W}$. From this, it follows that for any $0<s<t$,
\begin{equation*}
\begin{split}
f^{**}(s)-f^{**}(t)&=\int_{s}^t-\frac{d}{d\tau}\big(f^{**}(\tau)\big)\,d\tau\\
&\leq\int_{s}^t\frac{\|f\|_{\mathcal{W}}}{\tau}\,d\tau\\
&=\|f\|_{\mathcal{W}}\cdot\big[\log t-\log s\big]=\|f\|_{\mathcal{W}}\cdot\log\frac{t}{\,s\,},
\end{split}
\end{equation*}
which in turn implies that
\begin{equation}\label{key2}
f^{**}(s)\leq f^{**}(t)+\|f\|_{\mathcal{W}}\cdot\log\frac{t}{\,s\,},
\end{equation}
whenever $0<s<t$. For any $1<q<\infty$, in view of \eqref{equality2}, we can write
\begin{equation*}
\begin{split}
\|f\|_{L^q}&=\bigg(\int_0^{\infty}\big[f^{*}(s)\big]^qds\bigg)^{1/q}\\
&\leq\bigg(\int_0^{t}\big[f^{*}(s)\big]^qds\bigg)^{1/q}+
\bigg(\int_{t}^{\infty}\big[f^{*}(s)\big]^qds\bigg)^{1/q}\\
&:=J_1(t)+J_2(t),
\end{split}
\end{equation*}
where $t$ is a positive constant to be fixed below. Let us estimate the first term $J_1(t)$. By using \eqref{geq}, \eqref{key2} and Minkowski's inequality, we can deduce that
\begin{equation*}
\begin{split}
J_1(t)&\leq\bigg(\int_0^{t}\big[f^{**}(s)\big]^qds\bigg)^{1/q}\\
&\leq\bigg(\int_0^{t}\big[f^{**}(t)\big]^qds\bigg)^{1/q}
+\bigg(\int_0^{t}\Big[\|f\|_{\mathcal{W}}\cdot\log\frac{t}{\,s\,}\Big]^qds\bigg)^{1/q}\\
&=t^{1/q}\cdot f^{**}(t)+\|f\|_{\mathcal{W}}
\cdot\bigg(\int_0^{t}\Big[\log\frac{t}{\,s\,}\Big]^qds\bigg)^{1/q}.\\
\end{split}
\end{equation*}
Recall that the Stirling formula
\begin{equation*}
\Gamma(q+1)\approx\sqrt{2\pi}q^{q+\frac{\,1\,}{2}}\exp(-q)=\sqrt{2\pi q}\Big(\frac{q}{\,e\,}\Big)^q
\end{equation*}
holds as $q\to\infty$. Hence, we find that as $q\to\infty$,
\begin{equation}\label{aso}
\Gamma(q+1)^{1/q}=\mathcal{O}(q),
\end{equation}
that is, $\Gamma(q+1)^{1/q}\leq C\cdot q$, when $q$ is large. Here $C>0$ is an absolute constant independent of $q$. This fact, together with \eqref{weneed} and \eqref{key1}, gives us that
\begin{equation*}
\begin{split}
J_1(t)&\leq C\cdot q\Big[t^{1/q-1/p}\cdot\|f\|_{L^{p,r}}+\|f\|_{\mathcal{W}}\cdot t^{1/q}\Big],
\end{split}
\end{equation*}
where $C>0$ is independent of $f$ and $q>1$. We now estimate the second term $J_2(t)$.

$(1)$ When $r<\infty$ and $p\leq r<q<\infty$, since $f^{*}$ is non-increasing and $r/p-1\geq0$, we have
\begin{equation*}
\begin{split}
J_2(t)&=\bigg(\int_{t}^{\infty}\big[s^{1/p-1/r}f^{*}(s)\big]^r\cdot f^{*}(s)^{q-r}s^{1-r/p}ds\bigg)^{1/q}\\
&\leq\bigg(\int_{t}^{\infty}\big[s^{1/p-1/r}f^{*}(s)\big]^rds\bigg)^{1/q}\cdot \Big[f^{*}(t)^{q-r}t^{1-r/p}\Big]^{1/q}\\
&\leq\bigg(\int_{0}^{\infty}\big[s^{1/p}f^{*}(s)\big]^r\frac{ds}{s}\bigg)^{1/q}\cdot \Big[f^{*}(t)^{q-r}t^{1-r/p}\Big]^{1/q}.
\end{split}
\end{equation*}
Moreover, by \eqref{geq}, \eqref{key1} and the fact that $q-r>0$, we thus obtain
\begin{equation*}
\begin{split}
J_2(t)&\leq\big(\|f\|_{L^{p,r}}\big)^{r/q}
\cdot \Big[f^{**}(t)^{q-r}t^{1-r/p}\Big]^{1/q}\\
&\leq \big(\|f\|_{L^{p,r}}\big)^{r/q}
\cdot \Big[\Big(\frac{p'}{r'}\Big)^{1/{r'}}\Big]^{1-r/q}
t^{1/q-1/p}\big(\|f\|_{L^{p,r}}\big)^{1-r/q}\\
&=\Big[\Big(\frac{p'}{r'}\Big)^{1/{r'}}\Big]^{1-r/q}
t^{1/q-1/p}\cdot\|f\|_{L^{p,r}}.
\end{split}
\end{equation*}
$(2)$ When $r=\infty$ and $p<q<\infty$, in this case, we have
\begin{equation*}
\begin{split}
J_2(t)&=\bigg(\int_{t}^{\infty}\big[s^{1/p}f^{*}(s)\big]^{q}\frac{1}{s^{q/p}}ds\bigg)^{1/q}\\
&\leq\Big[\sup_{s>0}s^{1/p}f^{*}(s)\Big]\cdot\bigg(\int_{t}^{\infty}\frac{1}{s^{q/p}}ds\bigg)^{1/q}\\
&=\Big(\frac{p}{q-p}\Big)^{1/q} t^{1/q-1/p}\cdot\|f\|_{L^{p,\infty}},
\end{split}
\end{equation*}
where the last integral is convergent since $q/p>1$. Notice that $p'\geq r'$. Thus, for all $q$ with $p<q<\infty$,
\begin{equation*}
\Big[\Big(\frac{p'}{r'}\Big)^{1/{r'}}\Big]^{1-r/q}\leq\Big(\frac{p'}{r'}\Big)^{1/{r'}}.
\end{equation*}
Moreover, it is easy to see that
\begin{equation*}
\lim_{q\to\infty}\Big(\frac{p}{q-p}\Big)^{1/q}=1.\Longrightarrow \Big(\frac{p}{q-p}\Big)^{1/q}=\mathcal{O}(1)\quad \mbox{as}~ q\to\infty.
\end{equation*}
Summing up the above estimates, we then conclude that
\begin{equation*}
\|f\|_{L^q}\lesssim q\Big[t^{1/q-1/p}\cdot\|f\|_{L^{p,r}}+\|f\|_{\mathcal{W}}\cdot t^{1/q}\Big]
\end{equation*}
with the implicit positive constant independent of $q$ (may depend on $p$ and $r$). We now take $t>0$ in this estimate such that
\begin{equation*}
t^{1/q-1/p}\cdot\|f\|_{L^{p,r}}=\|f\|_{\mathcal{W}}\cdot t^{1/q}.
\end{equation*}
In other words,
\begin{equation*}
t:=\left(\frac{\|f\|_{L^{p,r}}}{\|f\|_{\mathcal{W}}}\right)^p.
\end{equation*}
Therefore, by such choice of $t>0$, we finally obtain
\begin{equation*}
\|f\|_{L^q}\lesssim q\cdot\big(\|f\|_{L^{p,r}}\big)^{p/q}\cdot\big(\|f\|_{\mathcal{W}}\big)^{1-p/q},
\end{equation*}
where the implicit positive constant is independent of $q>1$. We are done.
\end{proof}

\begin{rem}\label{citeremark}
(1) In particular, when $1<p<\infty$ and $r=p$ in Theorem \ref{strong},
\begin{equation}\label{remarks1}
L^{p,p}(\mathbb R^n)\cap \mathcal{W}=L^p(\mathbb R^n)\cap \mathcal{W}\hookrightarrow L^q(\mathbb R^n),
\end{equation}
for all $q$ with $p<q<\infty$. Furthermore, in view of Lemma \ref{lemmaw}, we have the following continuous embedding
\begin{equation}\label{remarks2}
L^p(\mathbb R^n)\cap \mathrm{BMO}(\mathbb R^n)\hookrightarrow L^q(\mathbb R^n),
\end{equation}
for all $q$ with $p<q<\infty$, where $1<p<\infty$, and
\begin{equation}\label{remarks3}
\|f\|_{L^q}\lesssim q\cdot\big(\|f\|_{L^{p}}\big)^{p/q}\cdot\big(\|f\|_{\mathrm{BMO}}\big)^{1-p/q}
\end{equation}
with the implicit positive constant independent of $q>1$. The above embedding \eqref{remarks2} and the interpolation inequality \eqref{remarks3} were first proved by Chen--Zhu in \cite[Theorem 2]{Chen} as far as we know, by using the technique of the non-increasing rearrangement and a slight variant of the John--Nirenberg inequality. Moreover, Chen--Zhu showed that in \eqref{remarks2} the number $p$ is allowed to be 1, that is, they showed that for any $1\leq p<q<\infty$,
\begin{equation}\label{chen}
\|f\|_{L^q}\leq C(n,p)q\cdot\big(\|f\|_{L^p}\big)^{p/q}\cdot\big(\|f\|_{\mathrm{BMO}}\big)^{1-p/q}
\end{equation}
holds true for all $f\in L^p(\mathbb R^n)\cap \mathrm{BMO}(\mathbb R^n)$, where $C(n,p)$ is a positive constant depending only on the dimension $n$ and $p$. Later, this estimate was improved and extended by many authors through different methods,
see, for example, \cite[Theorems 2.1 and 2.2]{Kozono2}, \cite[Theorem 1.1]{wadade} and \cite[Lemma 3]{mil}. It is worth mentioning that \eqref{chen} was recently reproved by the author in \cite{wang}, by using the classical Calderon--Zygmund decomposition. The above embedding \eqref{remarks1} was first established by Kozono--Wadade in \cite[Lemmas 3.2 and 3.3]{Kozono2}, where the number $p$ is also allowed to be 1.

(2) It should be pointed out that the growth order $q$ as $q\to\infty$ in \eqref{chen} is optimal
(see \cite[Theorem 2.2 and Remark (iii)]{Kozono2}).

(3) Inspired by the works in \cite{Chen} and \cite{Kozono2}, it is very natural to ask whether there is a corresponding estimate for the case $p=1$ in Theorem \ref{strong} or Corollary \ref{cor33}. Our method cannot be used to establish such a result because the key estimate \eqref{key1} does not hold if $p=1$.
\end{rem}

\subsection{Related bilinear estimates and John--Nirenberg type inequalities}
In this section, let us give some applications. Let $1<p<\infty$ and $p\leq r\leq\infty$. A direct consequence of Corollary \ref{cor33} is the following bilinear estimate:
\begin{equation}\label{bilinear1}
\|\mathcal{F}\cdot \mathcal{G}\|_{L^{p}}\lesssim p^2
\Big[\|\mathcal{F}\|_{L^{p,r}}\|\mathcal{G}\|_{\mathrm{BMO}}+\|\mathcal{G}\|_{L^{p,r}}\|\mathcal{F}\|_{\mathrm{BMO}}\Big]
\end{equation}
holds for two functions $\mathcal{F},\mathcal{G}\in L^{p,r}(\mathbb R^n)\cap \mathrm{BMO}(\mathbb R^n)$, where the implicit positive constant is independent of both $p$ and the functions $\mathcal{F},\mathcal{G}$. As a matter of fact, for any $p>1$, H\"{o}lder's inequality and Corollary \ref{cor33} with $q=2p$ yield
\begin{equation*}
\begin{split}
\|\mathcal{F}\cdot \mathcal{G}\|_{L^{p}}&\leq\|\mathcal{F}\|_{L^{2p}}\cdot\|\mathcal{G}\|_{L^{2p}}\\
&\lesssim p^2\Big[\|\mathcal{F}\|^{1/2}_{L^{p,r}}\|\mathcal{F}\|^{1/2}_{\mathrm{BMO}}
\cdot\|\mathcal{G}\|^{1/2}_{L^{p,r}}\|\mathcal{G}\|^{1/2}_{\mathrm{BMO}}\Big]\\
&=p^2\Big[\|\mathcal{F}\|_{L^{p,r}}\|\mathcal{G}\|_{\mathrm{BMO}}
\cdot\|\mathcal{G}\|_{L^{p,r}}\|\mathcal{F}\|_{\mathrm{BMO}}\Big]^{1/2}.
\end{split}
\end{equation*}
Moreover, from the elementary inequality $2\sqrt{AB}\leq A+B$, $A,B>0$, it follows that
\begin{equation*}
\begin{split}
\|\mathcal{F}\cdot \mathcal{G}\|_{L^{p}}
&\lesssim p^2\Big[\|\mathcal{F}\|_{L^{p,r}}\|\mathcal{G}\|_{\mathrm{BMO}}+\|\mathcal{G}\|_{L^{p,r}}\|\mathcal{F}\|_{\mathrm{BMO}}\Big],
\end{split}
\end{equation*}
as desired. In particular, when $p=r$, then we have
\begin{equation*}
\|\mathcal{F}\cdot \mathcal{G}\|_{L^{p}}\leq C(p,n)
\Big[\|\mathcal{F}\|_{L^{p}}\|\mathcal{G}\|_{\mathrm{BMO}}+\|\mathcal{G}\|_{L^{p}}\|\mathcal{F}\|_{\mathrm{BMO}}\Big]
\end{equation*}
for two functions $\mathcal{F},\mathcal{G}\in L^{p}(\mathbb R^n)\cap \mathrm{BMO}(\mathbb R^n)$ with $1<p<\infty$. These bilinear estimates for $1<p<\infty$ were proved by Kozono and Taniuchi using the boundedness of bilinear pseudo-differential operators in BMO
(due to Coifman and Meyer), see \cite[Lemma 1]{Kozono}(see also \cite{nakai3} for the dyadic BMO case, as an extension of Kozono and Taniuchi's result on the usual BMO). Such kind of inequalities are used to extend some results on uniqueness and regularity of weak solutions to the Navier--Stokes equations. Similarly, for three functions $\mathcal{F},\mathcal{G},\mathcal{H}\in L^{p,r}(\mathbb R^n)\cap\mathrm{BMO}(\mathbb R^n)$ with $1<p<\infty$ and $p\leq r\leq\infty$, by H\"{o}lder's inequality and Corollary \ref{cor33} with $q=3p$, we can deduce that
\begin{equation*}
\begin{split}
&\|\mathcal{F}\cdot \mathcal{G}\cdot \mathcal{H}\|_{L^p}\\
&\leq\|\mathcal{F}\|_{L^{3p}}\cdot\|\mathcal{G}\|_{L^{3p}}\cdot\|\mathcal{H}\|_{L^{3p}}\\
&\lesssim p^3
\Big(\|\mathcal{F}\|_{L^{p,r}}^{1/3}\|\mathcal{F}\|_{\mathrm{BMO}}^{2/3}\cdot\|\mathcal{G}\|_{L^{p,r}}^{1/3}\|\mathcal{G}\|_{\mathrm{BMO}}^{2/3}
\cdot\|\mathcal{H}\|_{L^{p,r}}^{1/3}\|\mathcal{H}\|_{\mathrm{BMO}}^{2/3}\Big)\\
&=p^3\Big(\|\mathcal{F}\|_{L^{p,r}}\|\mathcal{G}\|_{\mathrm{BMO}}\|\mathcal{H}\|_{\mathrm{BMO}}
\cdot\|\mathcal{G}\|_{L^{p,r}}\|\mathcal{F}\|_{\mathrm{BMO}}\|\mathcal{H}\|_{\mathrm{BMO}}\\
&\cdot\|\mathcal{H}\|_{L^{p,r}}\|\mathcal{F}\|_{\mathrm{BMO}}\|\mathcal{G}\|_{\mathrm{BMO}}\Big)^{1/3}.
\end{split}
\end{equation*}
Hence, we employ the elementary inequality $3(ABC)^{1/3}\leq A+B+C$ to obtain
\begin{equation*}
\begin{split}
\|\mathcal{F}\cdot \mathcal{G}\cdot \mathcal{H}\|_{L^p}
&\lesssim p^3\Big(\|\mathcal{F}\|_{L^{p,r}}\|\mathcal{G}\|_{\mathrm{BMO}}\|\mathcal{H}\|_{\mathrm{BMO}}
+\|\mathcal{G}\|_{L^{p,r}}\|\mathcal{F}\|_{\mathrm{BMO}}\|\mathcal{H}\|_{\mathrm{BMO}}\\
&+\|\mathcal{H}\|_{L^{p,r}}\|\mathcal{F}\|_{\mathrm{BMO}}\|\mathcal{G}\|_{\mathrm{BMO}}\Big).
\end{split}
\end{equation*}
In general, for the case $3\leq K\in \mathbb{N}$, the following estimate
\begin{equation*}
\begin{split}
&\bigg\|\prod_{j=1}^K \mathcal{F}_j\bigg\|_{L^p}\\
&\lesssim p^{K}
\sum_{j=1}^{K}\|\mathcal{F}_1\|_{\mathrm{BMO}}\cdots\|\mathcal{F}_{j-1}\|_{\mathrm{BMO}}
\|\mathcal{F}_j\|_{L^{p,r}}\|\mathcal{F}_{j+1}\|_{\mathrm{BMO}}\cdots\|\mathcal{F}_K\|_{\mathrm{BMO}}
\end{split}
\end{equation*}
holds for all $\mathcal{F}_1,\dots,\mathcal{F}_{K}\in L^{p,r}(\mathbb R^n)\cap\mathrm{BMO}(\mathbb R^n)$ with $1<p<\infty$ and $p\leq r\leq\infty$.

For $p\geq1$, we set
\begin{equation}\label{phif}
\Phi_p(t):=\exp(t)-\sum_{j=0}^{p-1}\frac{t^j}{j!}=\sum_{j=p}^{\infty}\frac{t^j}{j!},\quad t>0.
\end{equation}
As another application of Corollary \ref{cor33}, we have the following result, which may be regarded as a generalization of the John--Nirenberg inequality to the Lorentz spaces. There exists a positive constant $C({n,p,r,\alpha})$ depending on $n,p,r$ and $\alpha$ such that
\begin{equation}\label{bilinear14}
\int_{\mathbb R^n}\bigg[\Phi_p\Big(\alpha\frac{|f(x)|}{\|f\|_{\mathrm{BMO}}}\Big)\bigg]\,dx
\leq C({n,p,r,\alpha})\bigg(\frac{\|f\|_{L^{p,r}}}{\|f\|_{\mathrm{BMO}}}\bigg)^p
\end{equation}
holds for all $f\in L^{p,r}(\mathbb R^n)\cap\mathrm{BMO}(\mathbb R^n)$ with $1<p<\infty$ and $p\leq r\leq\infty$, provided that $0<\alpha<\alpha_n$, where $\alpha_n$ is a positive constant which depends on $n,r$ and $p$.

Since for any $p\leq j<\infty$, we have that
\begin{equation}\label{wanguseinequ1}
\int_{\mathbb R^n}|f(x)|^j\,dx=\big\|f\big\|^j_{L^{j}}\leq {C}(n,p,r)^jj^j\cdot\big(\|f\|_{L^{p,r}}\big)^p\cdot\big(\|f\|_{\mathrm{BMO}}\big)^{j-p}
\end{equation}
holds for all $f\in L^{p,r}(\mathbb R^n)\cap\mathrm{BMO}(\mathbb R^n)$, where ${C}(n,p,r)$ is a positive constant depending on $n,r$ and $p$, but not on $j$. Therefore, by interchanging the order of integration and summation and using the estimate \eqref{wanguseinequ1}, we get
\begin{equation*}
\begin{split}
\int_{\mathbb R^n}\bigg[\Phi_p\Big(\alpha\frac{|f(x)|}{\|f\|_{\mathrm{BMO}}}\Big)\bigg]\,dx
&=\int_{\mathbb R^n}\sum_{j=p}^{\infty}\frac{\alpha^j}{j!}\bigg(\frac{|f(x)|}{\|f\|_{\mathrm{BMO}}}\bigg)^jdx\\
&=\sum_{j=p}^{\infty}\frac{\alpha^j}{j!}\cdot\frac{1}{\|f\|^j_{\mathrm{BMO}}}\int_{\mathbb R^n}|f(x)|^j\,dx\\
&\leq \sum_{j=p}^{\infty}\frac{\alpha^j}{j!}\cdot\frac{1}{\|f\|^j_{\mathrm{BMO}}}\big[{C}(n,p,r)j\big]^{j}
\big(\|f\|_{L^{p,r}}\big)^p\cdot\big(\|f\|_{\mathrm{BMO}}\big)^{j-p}\\
&=\sum_{j=p}^{\infty}\frac{[{C}(n,p,r)\cdot\alpha]^j\cdot j^{j}}{j!}
\bigg(\frac{\|f\|_{L^{p,r}}}{\|f\|_{\mathrm{BMO}}}\bigg)^p.
\end{split}
\end{equation*}
We set $\gamma_j:=j^j/{(j!)}$. Since
\begin{equation*}
\lim_{j\to\infty}\frac{\gamma_{j+1}}{\gamma_j}=\lim_{j\to\infty}\left(1+\frac{\,1\,}{j}\right)^{j}=e,
\end{equation*}
we know that the above series converges provided ${C}(n,p,r)\cdot\alpha<e^{-1}$. That is, for
\begin{equation*}
0<\alpha<\alpha_n:=\big[{C}(n,p,r)\cdot e\big]^{-1},
\end{equation*}
we see that \eqref{bilinear14} holds with $C({n,p,r,\alpha})$ is the sum of the above convergent series.
In particular, when $r=p$ (in this case, the range of $p$ is $[1,\infty)$, as noted in Remark \ref{citeremark}, and $\Phi_p(t)=\exp(t)-1$ if $p=1$), the following inequality
\begin{equation}\label{japenm}
\int_{\mathbb R^n}\bigg[\exp\Big(\alpha\frac{|f(x)|}{\|f\|_{\mathrm{BMO}}}\Big)-1\bigg]dx
\leq C({n,\alpha})\bigg(\frac{\|f\|_{L^1}}{\|f\|_{\mathrm{BMO}}}\bigg)
\end{equation}
holds for all $f\in L^{1}(\mathbb R^n)\cap\mathrm{BMO}(\mathbb R^n)$, see \cite[Corollary 2.1]{Kozono2}. A direct computation yields
\begin{equation*}
\begin{split}
&\int_{\mathbb R^n}\bigg[\exp\Big(\alpha\frac{|f(x)|}{\|f\|_{\mathrm{BMO}}}\Big)-1\bigg]dx\\
&=\int_{\mathbb R^n}\sum_{j=1}^{\infty}\frac{\alpha^j}{j!}\bigg(\frac{|f(x)|}{\|f\|_{\mathrm{BMO}}}\bigg)^jdx\\
&=\sum_{j=1}^{\infty}\frac{\alpha^j}{j!}\cdot\frac{1}{\|f\|^j_{\mathrm{BMO}}}\int_{\mathbb R^n}|f(x)|^j\,dx.
\end{split}
\end{equation*}
Moreover, for each fixed $j\in \mathbb{N}$, it is easy to see that for any $\lambda>0$,
\begin{equation*}
\begin{split}
\int_{\mathbb R^n}|f(x)|^j\,dx&\geq\int_{\{x\in \mathbb R^n:|f(x)|>\lambda\}}|f(x)|^j\,dx\\
&>\lambda^j\cdot m\big(\big\{x\in \mathbb R^n:|f(x)|>\lambda\big\}\big).
\end{split}
\end{equation*}
Hence
\begin{equation}\label{jninequality1}
\begin{split}
&\int_{\mathbb R^n}\bigg[\exp\Big(\alpha\frac{|f(x)|}{\|f\|_{\mathrm{BMO}}}\Big)-1\bigg]\,dx\\
&\geq\sum_{j=1}^{\infty}\frac{\alpha^j}{j!}\cdot\bigg(\frac{\lambda}{\|f\|_{\mathrm{BMO}}}\bigg)^j
\cdot m\big(\big\{x\in \mathbb R^n:|f(x)|>\lambda\big\}\big)\\
&=\Big[\exp\Big(\frac{\alpha\lambda}{\|f\|_{\mathrm{BMO}}}\Big)-1\Big]
\cdot m\big(\big\{x\in \mathbb R^n:|f(x)|>\lambda\big\}\big).
\end{split}
\end{equation}
Therefore, it follows from \eqref{japenm} and \eqref{jninequality1} that for any $\lambda>0$,
\begin{equation*}
\begin{split}
&\Big[\exp\Big(\frac{\alpha\lambda}{\|f\|_{\mathrm{BMO}}}\Big)-1\Big]
\cdot m\big(\big\{x\in \mathbb R^n:|f(x)|>\lambda\big\}\big)
\leq C({n,\alpha})\bigg(\frac{\|f\|_{L^1}}{\|f\|_{\mathrm{BMO}}}\bigg),
\end{split}
\end{equation*}
which in turn implies that
\begin{equation*}
\begin{split}
m\big(\big\{x\in \mathbb R^n:|f(x)|>\lambda\big\}\big)
\leq C({n,\alpha})\bigg(\frac{\|f\|_{L^1}}{\|f\|_{\mathrm{BMO}}}\bigg)
\cdot\Big[\exp\Big(\frac{\alpha\lambda}{\|f\|_{\mathrm{BMO}}}\Big)-1\Big]^{-1}.
\end{split}
\end{equation*}
If $\lambda>\|f\|_{\mathrm{BMO}}$, then
\begin{equation*}
\Big[\exp\Big(\frac{\alpha\lambda}{\|f\|_{\mathrm{BMO}}}\Big)-1\Big]^{-1}
\leq\exp\Big(-\frac{\alpha\lambda}{\|f\|_{\mathrm{BMO}}}\Big)
\cdot\frac{1}{1-e^{-\alpha}},
\end{equation*}
and hence
\begin{equation*}
\begin{split}
d_f(\lambda)&=m\big(\big\{x\in \mathbb R^n:|f(x)|>\lambda\big\}\big)\\
&\leq C({n,\alpha})\bigg(\frac{\|f\|_{L^1}}{\|f\|_{\mathrm{BMO}}}\bigg)\exp\Big(-\frac{\alpha\lambda}{\|f\|_{\mathrm{BMO}}}\Big)
\end{split}
\end{equation*}
whenever $\lambda>\|f\|_{\mathrm{BMO}}$. This interesting estimate is a variant of the John--Nirenberg inequality, which was given by Chen--Zhu in \cite[(4)]{Chen}.

\section{Interpolation inequalities between Lorentz--Morrey, $L^{\infty}$ and BMO spaces}
\subsection{Main results}
The main results of this section are stated as follows.

\begin{thm}\label{thmwang2}
Let $n\in \mathbb{N}$, $1\leq p<\infty$, $p<r\leq\infty$ and $0<\kappa<1$. Then the following statements hold.

$(1)$ If $r<\infty$ and $f\in LM^{p,r;\kappa}(\mathbb R^n)\cap{L^\infty}(\mathbb R^n)$, then for every $q$ with $r<q<\infty$, we have $f\in M^{q;\kappa}(\mathbb R^n)$, and the interpolation inequality
\begin{equation*}
\|f\|_{M^{q;\kappa}}\leq 2\cdot\big(\|f\|_{LM^{p,r;\kappa}}\big)^{p/q}\cdot\big(\|f\|_{L^\infty}\big)^{1-p/q}
\end{equation*}
holds for all $f\in LM^{p,r;\kappa}(\mathbb R^n)\cap{L^\infty}(\mathbb R^n)$.

$(2)$ If $r=\infty$ and $f\in LM^{p,r;\kappa}(\mathbb R^n)\cap{L^\infty}(\mathbb R^n)$, then for every $q$ with $p<q<\infty$, we have $f\in M^{q;\kappa}(\mathbb R^n)$, and the interpolation inequality
\begin{equation*}
\|f\|_{M^{q;\kappa}}\leq \Big[2^{1-1/q}\Big(\frac{q}{q-p}\Big)^{1/q}\Big]
\cdot\big(\|f\|_{LM^{p,\infty;\kappa}}\big)^{p/q}\cdot\big(\|f\|_{L^\infty}\big)^{1-p/q}
\end{equation*}
holds for all $f\in LM^{p,\infty;\kappa}(\mathbb R^n)\cap{L^\infty}(\mathbb R^n)$.
\end{thm}

\begin{thm}\label{weak}
Let $n\in \mathbb{N}$, $1<p<\infty$, $p\leq r\leq\infty$ and $0<\kappa<1$. If $p\leq r<\infty$ and $f\in LM^{p,r;\kappa}(\mathbb R^n)\cap \mathrm{BMO}(\mathbb R^n)$, then for every $q$ with $r<q<\infty$(or if $f\in LM^{p,\infty;\kappa}(\mathbb R^n)\cap \mathrm{BMO}(\mathbb R^n)$, then for every $q$ with $p<q<\infty$), we have $f\in M^{q;\kappa}(\mathbb R^n)$, and the following interpolation inequality
\begin{equation*}
\|f\|_{M^{q;\kappa}}\lesssim q\cdot\big(\|f\|_{LM^{p,r;\kappa}}\big)^{p/q}\cdot\big(\|f\|_{\mathrm{BMO}}\big)^{1-p/q}
\end{equation*}
holds for all $f\in LM^{p,r;\kappa}(\mathbb R^n)\cap{\mathrm{BMO}}(\mathbb R^n)$. Here the implicit positive constant is independent of $q$.
\end{thm}

\subsection{Proofs of Theorems $\ref{thmwang2}$ and $\ref{weak}$}
\begin{proof}[Proof of Theorem $\ref{thmwang2}$]
Let $f\in LM^{p,r;\kappa}(\mathbb R^n)\cap{L^\infty}(\mathbb R^n)$ with $1\leq p<\infty$, $p<r\leq\infty$ and $0<\kappa<1$. For any ball $\mathcal{B}$ in $\mathbb R^n$ and for $p<q<\infty$, by equation \eqref{equality1}, we see that
\begin{equation*}
\begin{split}
&\frac{1}{m(\mathcal{B})^{\kappa/q}}\|f\|_{L^{q}(\mathcal{B})}
=\frac{1}{m(\mathcal{B})^{\kappa/q}}\bigg(\int_0^{\infty}\big[f^{*}_{\mathcal{B}}(s)\big]^qds\bigg)^{1/q}\\
&\leq\frac{1}{m(\mathcal{B})^{\kappa/q}}\bigg(\int_0^{t}\big[f^{*}_{\mathcal{B}}(s)\big]^qds\bigg)^{1/q}
+\frac{1}{m(\mathcal{B})^{\kappa/q}}\bigg(\int_{t}^{\infty}\big[f^{*}_{\mathcal{B}}(s)\big]^qds\bigg)^{1/q}\\
&:=I'_1(t)+I'_2(t),
\end{split}
\end{equation*}
where $t>0$ is a constant to be determined later. Recall that for any fixed ball $\mathcal{B}\subset\mathbb R^n$,
\begin{equation}\label{often2}
f^{*}_{\mathcal{B}}(s)\leq f^{*}(s)\leq \|f\|_{L^{\infty}}=\sup_{t>0}f^{*}(t)
\end{equation}
holds for all $s>0$. Therefore, we have
\begin{equation*}
\begin{split}
I'_1(t)&\leq\frac{1}{m(\mathcal{B})^{\kappa/q}}\|f\|_{L^\infty}\cdot t^{1/q}.
\end{split}
\end{equation*}
$(1)$ When $r<\infty$ and $p<r<q<\infty$, in this case, it follows from \eqref{often2} that for any $t>0$,
\begin{equation*}
\begin{split}
I'_2(t)&=\frac{1}{m(\mathcal{B})^{\kappa/q}}
\bigg(\int_{t}^{\infty}s^{r/p-1}\big[f^{*}_{\mathcal{B}}(s)\big]^{r}
\cdot\big[f^{*}_{\mathcal{B}}(s)\big]^{q-r}\frac{1}{s^{r/p-1}}ds\bigg)^{1/q}\\
&\leq\frac{1}{m(\mathcal{B})^{\kappa/q}}
\bigg(\int_{t}^{\infty}\Big[s^{1/p}f^{*}_{\mathcal{B}}(s)\Big]^r\frac{ds}{s}\bigg)^{1/q}
\cdot\big(\|f\|_{L^\infty}\big)^{1-r/q}\Big(\frac{1}{t^{r/p-1}}\Big)^{1/q},
\end{split}
\end{equation*}
where in the last inequality we have used the facts that $q-r>0$ and $r/p-1>0$. Moreover, by the definition of $LM^{p,r;\kappa}(\mathbb R^n)$, we obtain
\begin{equation*}
\begin{split}
I'_2(t)&\leq\frac{1}{m(\mathcal{B})^{\kappa/q}}
\bigg(\int_0^{\infty}\Big[s^{1/p}f^{*}_{\mathcal{B}}(s)\Big]^r\frac{ds}{s}\bigg)^{1/q}
\cdot\big(\|f\|_{L^\infty}\big)^{1-r/q}\Big(\frac{1}{t^{r/p-1}}\Big)^{1/q}\\
&\leq\frac{1}{m(\mathcal{B})^{\kappa/q}}
\big(\|f\|_{LM^{p,r;\kappa}}m(\mathcal{B})^{\kappa/p}\big)^{r/q}
\big(\|f\|_{L^\infty}\big)^{1-r/q}\cdot\Big(\frac{1}{t^{r/p-1}}\Big)^{1/q}.
\end{split}
\end{equation*}
We now choose $t>0$ so that
\begin{equation*}
\big(\|f\|_{LM^{p,r;\kappa}}m(\mathcal{B})^{\kappa/p}\big)^{r/q}
\big(\|f\|_{L^\infty}\big)^{1-r/q}\cdot\Big(\frac{1}{t^{r/p-1}}\Big)^{1/q}=\|f\|_{L^\infty}\cdot t^{1/q},
\end{equation*}
that is,
\begin{equation*}
t:=\left(\frac{\|f\|_{LM^{p,r;\kappa}}}{\|f\|_{L^\infty}}\right)^pm(\mathcal{B})^{\kappa}.
\end{equation*}
Therefore,
\begin{equation*}
\begin{split}
\frac{1}{m(\mathcal{B})^{\kappa/q}}\|f\|_{L^{q}(\mathcal{B})}
&\leq 2\cdot\frac{1}{m(\mathcal{B})^{\kappa/q}}\|f\|_{L^\infty}\cdot t^{1/q}\\
&=2\cdot\big(\|f\|_{LM^{p,r;\kappa}}\big)^{p/q}\cdot\big(\|f\|_{L^\infty}\big)^{1-p/q}.
\end{split}
\end{equation*}
$(2)$ When $r=\infty$ and $p<q<\infty$, in this case, by the definition of $LM^{p,\infty;\kappa}(\mathbb R^n)$ and the fact that $q/p>1$, we can deduce that for any $t>0$,
\begin{equation*}
\begin{split}
I'_2(t)&=\frac{1}{m(\mathcal{B})^{\kappa/q}}\bigg(\int_{t}^{\infty}\big[s^{1/p}f^{*}_{\mathcal{B}}(s)\big]^{q}\frac{1}{s^{q/p}}ds\bigg)^{1/q}\\
&\leq\frac{1}{m(\mathcal{B})^{\kappa/q}}\Big[\sup_{s>0}s^{1/p}f^{*}_{\mathcal{B}}(s)\Big]
\cdot\bigg(\int_{t}^{\infty}\frac{1}{s^{q/p}}ds\bigg)^{1/q}\\
&\leq \frac{1}{m(\mathcal{B})^{\kappa/q}}
\big(\|f\|_{LM^{p,\infty;\kappa}}m(\mathcal{B})^{\kappa/p}\big)\cdot\Big(\frac{1}{q/p-1}\Big)^{1/q} t^{1/q-1/p}.
\end{split}
\end{equation*}
We can choose $t>0$ so that
\begin{equation*}
\|f\|_{LM^{p,\infty;\kappa}}m(\mathcal{B})^{\kappa/p}\cdot t^{1/q-1/p}=\|f\|_{L^\infty}\cdot t^{1/q},
\end{equation*}
that is,
\begin{equation*}
t:=\left(\frac{\|f\|_{LM^{p,\infty;\kappa}}}{\|f\|_{L^\infty}}\right)^pm(\mathcal{B})^{\kappa}.
\end{equation*}
Therefore, we conclude from \eqref{elemen} that
\begin{equation*}
\begin{split}
\frac{1}{m(\mathcal{B})^{\kappa/q}}\|f\|_{L^{q}(\mathcal{B})}
&\leq\Big[\Big(\frac{p}{q-p}\Big)^{1/q}+1\Big]\frac{1}{m(\mathcal{B})^{\kappa/q}}\|f\|_{L^\infty}\cdot t^{1/q}\\
&=\Big[\Big(\frac{p}{q-p}\Big)^{1/q}+1\Big]
\cdot\big(\|f\|_{LM^{p,\infty;\kappa}}\big)^{p/q}\cdot\big(\|f\|_{L^\infty}\big)^{1-p/q}\\
&\leq \Big[2^{1-1/q}\Big(\frac{q}{q-p}\Big)^{1/q}\Big]
\cdot\big(\|f\|_{LM^{p,\infty;\kappa}}\big)^{p/q}\cdot\big(\|f\|_{L^\infty}\big)^{1-p/q}.
\end{split}
\end{equation*}
By taking the supremum over all balls $\mathcal{B}$ in $\mathbb R^n$, we complete the proof of Theorem $\ref{thmwang2}$.
\end{proof}

\begin{proof}[Proof of Theorem $\ref{weak}$]
Let $f\in LM^{p,r;\kappa}(\mathbb R^n)\cap \mathrm{BMO}(\mathbb R^n)$ with $1<p<\infty$, $p\leq r\leq\infty$ and $0<\kappa<1$. First of all, from the previous estimate \eqref{key1}, we know that for any fixed ball $\mathcal{B}$ in $\mathbb R^n$,
\begin{equation}\label{key3}
\begin{split}
(f\cdot\mathbf{1}_{\mathcal{B}})^{**}(t)=f^{**}_{\mathcal{B}}(t)
&\leq\Big(\frac{p'}{r'}\Big)^{1/{r'}}\frac{1}{t^{1/p}}\cdot\|f\cdot\mathbf{1}_{\mathcal{B}}\|_{L^{p,r}}\\
&=\Big(\frac{p'}{r'}\Big)^{1/{r'}}\frac{1}{t^{1/p}}\cdot\|f\|_{L^{p,r}(\mathcal{B})}.
\end{split}
\end{equation}
Here $r'=1$ when $r=\infty$. Next we claim that for each fixed ball $\mathcal{B}\subset\mathbb R^n$ and $f\in\mathrm{BMO}(\mathbb R^n)$, then $f\cdot\mathbf{1}_{\mathcal{B}}\in \mathrm{BMO}(\mathbb R^n)$, and the following inequality holds:
\begin{equation}\label{diff}
\|f\cdot\mathbf{1}_{\mathcal{B}}\|_{\mathrm{BMO}}\leq 2\|f\|_{\mathrm{BMO}}.
\end{equation}
In fact, for any ball $\mathcal{R}\subset\mathbb R^n$, by a simple calculation, we have
\begin{equation*}
\begin{split}
&\frac{1}{m(\mathcal{R})}\int_{\mathcal{R}}\big|(f\cdot\mathbf{1}_{\mathcal{B}})(y)-(f\cdot\mathbf{1}_{\mathcal{B}})_{\mathcal{R}}\big|\,dy\\
&\leq\frac{1}{m(\mathcal{R})}\int_{\mathcal{R}}\big|(f\cdot\mathbf{1}_{\mathcal{B}})(y)-f_{\mathcal{R}}\big|\,dy
+\frac{1}{m(\mathcal{R})}\int_{\mathcal{R}}\big|f_{\mathcal{R}}-(f\cdot\mathbf{1}_{\mathcal{B}})_{\mathcal{R}}\big|\,dy\\
&\leq \frac{2}{m(\mathcal{R})}\int_{\mathcal{R}}\big|(f\cdot\mathbf{1}_{\mathcal{B}})(y)-f_{\mathcal{R}}\big|\,dy
=\frac{2}{m(\mathcal{R})}\int_{\mathcal{R}\cap \mathcal{B}}\big|f(y)-f_{\mathcal{R}}\big|\,dy\\
&\leq\frac{2}{m(\mathcal{R})}\int_{\mathcal{R}}\big|f(y)-f_{\mathcal{R}}\big|\,dy\leq 2\|f\|_{\mathrm{BMO}}.
\end{split}
\end{equation*}
This proves \eqref{diff} by taking the supremum over all the balls $\mathcal{R}\subset\mathbb R^n$. This fact, together with Lemma \ref{lemmaw}, implies that $f\cdot\mathbf{1}_{\mathcal{B}}\in \mathcal{W}$, and
\begin{equation}\label{key5}
\|f\cdot\mathbf{1}_{\mathcal{B}}\|_{\mathcal{W}}\lesssim\|f\cdot\mathbf{1}_{\mathcal{B}}\|_{\mathrm{BMO}}
\lesssim\|f\|_{\mathrm{BMO}}.
\end{equation}
It then follows from \eqref{important} and \eqref{key5} that for any $0<s<t$,
\begin{equation*}
\begin{split}
f^{**}_{\mathcal{B}}(s)-f^{**}_{\mathcal{B}}(t)&=\int_{s}^t-\frac{d}{d\tau}\big(f^{**}_{\mathcal{B}}(\tau)\big)\,d\tau\\
&=\int_{s}^t\frac{f^{**}_{\mathcal{B}}(\tau)-f^{*}_{\mathcal{B}}(\tau)}{\tau}\,d\tau
\lesssim\int_{s}^t\frac{\|f\|_{\mathrm{BMO}}}{\tau}\,d\tau\\
&=\|f\|_{\mathrm{BMO}}\cdot\big[\log t-\log s\big]=\|f\|_{\mathrm{BMO}}\cdot\log\frac{t}{\,s\,},
\end{split}
\end{equation*}
which in turn implies that
\begin{equation}\label{key6}
f^{**}_{\mathcal{B}}(s)\lesssim f^{**}_{\mathcal{B}}(t)+\|f\|_{\mathrm{BMO}}\cdot\log\frac{t}{\,s\,},
\end{equation}
whenever $0<s<t$. For any $p<q<\infty$ and for any ball $\mathcal{B}$ in $\mathbb R^n$, by equation \eqref{equality1}, we can write
\begin{equation*}
\begin{split}
&\frac{1}{m(\mathcal{B})^{\kappa/q}}\|f\|_{L^{q}(\mathcal{B})}
=\frac{1}{m(\mathcal{B})^{\kappa/q}}\bigg(\int_0^{\infty}\big[f^{*}_{\mathcal{B}}(s)\big]^qds\bigg)^{1/q}\\
&\leq\frac{1}{m(\mathcal{B})^{\kappa/q}}\bigg(\int_0^{t}\big[f^{*}_{\mathcal{B}}(s)\big]^qds\bigg)^{1/q}
+\frac{1}{m(\mathcal{B})^{\kappa/q}}\bigg(\int_{t}^{\infty}\big[f^{*}_{\mathcal{B}}(s)\big]^qds\bigg)^{1/q}\\
&:=J'_1(t)+J'_2(t),
\end{split}
\end{equation*}
where $t>0$ is a constant to be fixed below. Let us estimate the first term $J'_1(t)$. Note that for any $s>0$,
\begin{equation}\label{fb}
f^{*}_{\mathcal{B}}(s)\leq f^{**}_{\mathcal{B}}(s)<+\infty.
\end{equation}
By Minkowski's inequality, \eqref{fb} and \eqref{key6}, we obtain that for all $t>0$,
\begin{equation*}
\begin{split}
J'_1(t)&\leq\frac{1}{m(\mathcal{B})^{\kappa/q}}\bigg(\int_0^{t}\big[f^{**}_{\mathcal{B}}(s)\big]^qds\bigg)^{1/q}\\
&\lesssim\frac{1}{m(\mathcal{B})^{\kappa/q}}\bigg(\int_0^{t}\big[f^{**}_{\mathcal{B}}(t)\big]^qds\bigg)^{1/q}
+\frac{1}{m(\mathcal{B})^{\kappa/q}}\bigg(\int_0^{t}\Big[\|f\|_{\mathrm{BMO}}\cdot\log\frac{t}{\,s\,}\Big]^qds\bigg)^{1/q}\\
&=\frac{1}{m(\mathcal{B})^{\kappa/q}}f^{**}_{\mathcal{B}}(t)\cdot t^{1/q}
+\frac{1}{m(\mathcal{B})^{\kappa/q}}\|f\|_{\mathrm{BMO}}\cdot\bigg(\int_0^{t}\Big[\log\frac{t}{\,s\,}\Big]^qds\bigg)^{1/q}.
\end{split}
\end{equation*}
Furthermore, it follows from \eqref{key3}, \eqref{weneed} and \eqref{aso} that
\begin{equation*}
\begin{split}
J'_1(t)&\lesssim \frac{1}{m(\mathcal{B})^{\kappa/q}}t^{1/q-1/p}\cdot\|f\|_{L^{p,r}(\mathcal{B})}
+ \frac{1}{m(\mathcal{B})^{\kappa/q}}\Gamma(q+1)^{1/q}\|f\|_{\mathrm{BMO}}\cdot t^{1/q}\\
&\lesssim  \frac{1}{m(\mathcal{B})^{\kappa/q}}
\cdot q\Big[t^{1/q-1/p}\cdot\|f\|_{L^{p,r}(\mathcal{B})}+\|f\|_{\mathrm{BMO}}\cdot t^{1/q}\Big]
\end{split}
\end{equation*}
with the implicit positive constant independent of $q$. Next we estimate the second term $J'_2(t)$.

$(1)$ When $r<\infty$ and $p\leq r<q<\infty$, since $f^{*}_{\mathcal{B}}$ is non-increasing and $r/p-1\geq0$, we get
\begin{equation*}
\begin{split}
J'_2(t)&=\frac{1}{m(\mathcal{B})^{\kappa/q}}\bigg(\int_{t}^{\infty}\big[s^{1/p-1/r}f^{*}_{\mathcal{B}}(s)\big]^r
\cdot f^{*}_{\mathcal{B}}(s)^{q-r}s^{1-r/p}ds\bigg)^{1/q}\\
&\leq\frac{1}{m(\mathcal{B})^{\kappa/q}}\bigg(\int_{t}^{\infty}\big[s^{1/p-1/r}f^{*}_{\mathcal{B}}(s)\big]^rds\bigg)^{1/q}
\cdot \Big[f^{*}_{\mathcal{B}}(t)^{q-r}t^{1-r/p}\Big]^{1/q}\\
&\leq\frac{1}{m(\mathcal{B})^{\kappa/q}}\bigg(\int_{0}^{\infty}\big[s^{1/p}f^{*}_{\mathcal{B}}(s)\big]^r\frac{ds}{s}\bigg)^{1/q}
\cdot \Big[f^{*}_{\mathcal{B}}(t)^{q-r}t^{1-r/p}\Big]^{1/q}.
\end{split}
\end{equation*}
Moreover, by using \eqref{fb}, \eqref{key3} and the fact that $q-r>0$, we thus obtain
\begin{equation*}
\begin{split}
J'_2(t)&\leq\frac{1}{m(\mathcal{B})^{\kappa/q}}\big(\|f\|_{L^{p,r}(\mathcal{B})}\big)^{r/q}
\cdot \Big[f^{**}_{\mathcal{B}}(t)^{q-r}t^{1-r/p}\Big]^{1/q}\\
&\leq \frac{1}{m(\mathcal{B})^{\kappa/q}}\big(\|f\|_{L^{p,r}(\mathcal{B})}\big)^{r/q}
\cdot \Big[\Big(\frac{p'}{r'}\Big)^{1/{r'}}\Big]^{1-r/q}
t^{1/q-1/p}\big(\|f\|_{L^{p,r}(\mathcal{B})}\big)^{1-r/q}\\
&=\frac{1}{m(\mathcal{B})^{\kappa/q}}\Big[\Big(\frac{p'}{r'}\Big)^{1/{r'}}\Big]^{1-r/q}
t^{1/q-1/p}\cdot\|f\|_{L^{p,r}(\mathcal{B})}.
\end{split}
\end{equation*}
$(2)$ When $r=\infty$ and $p<q<\infty$, in this case, we get
\begin{equation*}
\begin{split}
J'_2(t)&=\frac{1}{m(\mathcal{B})^{\kappa/q}}\bigg(\int_{t}^{\infty}
\big[s^{1/p}f^{*}_{\mathcal{B}}(s)\big]^{q}\frac{1}{s^{q/p}}ds\bigg)^{1/q}\\
&\leq\frac{1}{m(\mathcal{B})^{\kappa/q}}\Big[\sup_{s>0}s^{1/p}f^{*}_{\mathcal{B}}(s)\Big]
\cdot\bigg(\int_{t}^{\infty}\frac{1}{s^{q/p}}ds\bigg)^{1/q}\\
&=\frac{1}{m(\mathcal{B})^{\kappa/q}}\Big(\frac{p}{q-p}\Big)^{1/q} t^{1/q-1/p}\cdot\|f\|_{L^{p,\infty}(\mathcal{B})},
\end{split}
\end{equation*}
where the last integral converges since $q/p>1$. As mentioned before, for all $q$ with $p<q<\infty$,
\begin{equation*}
\Big[\Big(\frac{p'}{r'}\Big)^{1/{r'}}\Big]^{1-r/q}<\Big(\frac{p'}{r'}\Big)^{1/{r'}}\quad \&\quad
\Big(\frac{p}{q-p}\Big)^{1/q}=\mathcal{O}(1)\quad \mbox{as}~ q\to\infty.
\end{equation*}
Consequently, by the definition of $LM^{p,r;\kappa}(\mathbb R^n)$ with $p\leq r\leq\infty$, we conclude that for any ball $\mathcal{B}$ in $\mathbb R^n$,
\begin{equation*}
\begin{split}
&\frac{1}{m(\mathcal{B})^{\kappa/q}}\|f\|_{L^{q}(\mathcal{B})}\\
&\lesssim \frac{1}{m(\mathcal{B})^{\kappa/q}}
\cdot q\Big[t^{1/q-1/p}\cdot\|f\|_{L^{p,r}(\mathcal{B})}+\|f\|_{\mathrm{BMO}}\cdot t^{1/q}\Big]\\
&\lesssim\frac{1}{m(\mathcal{B})^{\kappa/q}}
\cdot q\Big[t^{1/q-1/p}\cdot\|f\|_{LM^{p,r;\kappa}}m(\mathcal{B})^{\kappa/p}+\|f\|_{\mathrm{BMO}}\cdot t^{1/q}\Big],
\end{split}
\end{equation*}
where the implicit positive constant is independent of $q$. We now take $t>0$ in this estimate so that
\begin{equation*}
t^{1/q-1/p}\cdot\|f\|_{LM^{p,r;\kappa}}m(\mathcal{B})^{\kappa/p}=\|f\|_{\mathrm{BMO}}\cdot t^{1/q}.
\end{equation*}
In other words,
\begin{equation*}
t:=\left(\frac{\|f\|_{LM^{p,r;\kappa}}}{\|f\|_{\mathrm{BMO}}}\right)^pm(\mathcal{B})^{\kappa}.
\end{equation*}
Therefore, by such choice of $t>0$,
\begin{equation*}
\begin{split}
\frac{1}{m(\mathcal{B})^{\kappa/q}}\|f\|_{L^{q}(\mathcal{B})}
&\lesssim\frac{1}{m(\mathcal{B})^{\kappa/q}}\cdot q\|f\|_{\mathrm{BMO}}\cdot t^{1/q}\\
&\leq q\cdot\big(\|f\|_{LM^{p,r;\kappa}}\big)^{p/q}\cdot\big(\|f\|_{\mathrm{BMO}}\big)^{1-p/q}.
\end{split}
\end{equation*}
Finally, by taking the supremum over all balls $\mathcal{B}$ in $\mathbb R^n$, we complete the proof of Theorem \ref{weak}.
\end{proof}

\begin{rem}\label{citeremark3}
(1) Theorem $\ref{thmwang2}$ tells us that the embedding constant for the case $p<r<\infty$ is 2 and the asymptotically embedding constant for the case $p<r=\infty$ is also 2 by means of \eqref{22}.

(2) In particular, by putting $r=p$ in Theorem \ref{weak}, we immediately have
\begin{equation*}
LM^{p,p;\kappa}(\mathbb R^n)\cap \mathrm{BMO}(\mathbb R^n)
=M^{p;\kappa}(\mathbb R^n)\cap \mathrm{BMO}(\mathbb R^n)\hookrightarrow  M^{q;\kappa}(\mathbb R^n)
\end{equation*}
for all $q$ with $p<q<\infty$, where $1<p<\infty$ and $0<\kappa<1$. Moreover,
\begin{equation*}
\|f\|_{M^{q;\kappa}}\lesssim q\cdot\big(\|f\|_{M^{p;\kappa}}\big)^{p/q}\cdot\big(\|f\|_{\mathrm{BMO}}\big)^{1-p/q}
\end{equation*}
with the implicit positive constant independent of $q$. Actually, the above result is also true for the endpoint case $p=1$, which was recently proved by the author in \cite{wang} using a local version of Calderon--Zygmund decomposition. That is, it can be shown that for any $1\leq p<q<\infty$ and $0<\kappa<1$,
\begin{equation}\label{chenwang}
\|f\|_{M^{q;\kappa}}\leq C(n,p)q\cdot\big(\|f\|_{M^{p;\kappa}}\big)^{p/q}\cdot\big(\|f\|_{\mathrm{BMO}}\big)^{1-p/q}
\end{equation}
holds true for all $f\in M^{p;\kappa}(\mathbb R^n)\cap \mathrm{BMO}(\mathbb R^n)$, where $C(n,p)$ is a positive constant depending only on the dimension $n$ and $p$. This leads to a natural conjecture that whether $p=1$ is included in Theorem \ref{weak}. Unfortunately, the key estimate \eqref{key3} fails when $p=1$. So we need to look for a new method. Furthermore, in analogy with \eqref{chen}
(see \cite[Theorem 2.2 and Remark (iii)]{Kozono2}), we believe that the (linear) growth order $q$ in Theorem \ref{weak} is best possible as $q\to\infty$.

(3) Inspired by Theorem \ref{strong}, a natural question appearing here is whether the space $LM^{p,r;\kappa}(\mathbb R^n)\cap \mathrm{BMO}(\mathbb R^n)$ in Theorem \ref{weak} can be replaced by $LM^{p,r;\kappa}(\mathbb R^n)\cap \mathcal{W}$. The main difficulty is to show that for each fixed ball $\mathcal{B}\subset\mathbb R^n$ and $f\in \mathcal{W}$, then $f\cdot\mathbf{1}_{\mathcal{B}}\in \mathcal{W}$, and
\begin{equation*}
\|f\cdot\mathbf{1}_{\mathcal{B}}\|_{\mathcal{W}}\lesssim \|f\|_{\mathcal{W}}.
\end{equation*}
Unfortunately, the key technique in \eqref{diff} fails for the space $\mathcal{W}$. In this situation, we need to look for a new method.
\end{rem}

\subsection{Related bilinear estimates and John--Nirenberg type inequalities}
Let us now proceed with some applications. Let $1<p<\infty$, $0<\kappa_1,\kappa_2<1$ and $p\leq r\leq\infty$. As a consequence of Theorem \ref{weak}, we can show that the following bilinear estimate
\begin{equation}\label{bilinear21}
\|\mathcal{F}\cdot \mathcal{G}\|_{M^{p;\kappa}}\lesssim p^2
\Big(\|\mathcal{F}\|_{LM^{p,r;\kappa_1}}\|\mathcal{G}\|_{\mathrm{BMO}}
+\|\mathcal{G}\|_{LM^{p,r;\kappa_2}}\|\mathcal{F}\|_{\mathrm{BMO}}\Big)
\end{equation}
holds for two functions $\mathcal{F}\in LM^{p,r;\kappa_1}(\mathbb R^n)\cap\mathrm{BMO}(\mathbb R^n)$ and $\mathcal{G}\in LM^{p,r;\kappa_2}(\mathbb R^n)\cap\mathrm{BMO}(\mathbb R^n)$, where $\kappa={(\kappa_1+\kappa_2)}/2$ and the implicit positive constant is independent of both $p$ and the functions $\mathcal{F},\mathcal{G}$. In fact, by using H\"older's inequality, we can deduce that for any given ball $\mathcal{B}$ in $\mathbb R^n$,
\begin{equation*}
\begin{split}
\bigg(\int_{\mathcal{B}}|\mathcal{F}(y)\cdot \mathcal{G}(y)|^p\,dy\bigg)^{1/p}
&\leq\bigg(\int_{\mathcal{B}}|\mathcal{F}(y)|^{2p}dy\bigg)^{1/{2p}}
\cdot\bigg(\int_{\mathcal{B}}|\mathcal{G}(y)|^{2p}dy\bigg)^{1/{2p}}\\
&\leq\Big(\|\mathcal{F}\|_{M^{2p;\kappa_1}}m(\mathcal{B})^{{\kappa_1}/{2p}}\Big)
\cdot\Big(\|\mathcal{G}\|_{M^{2p;\kappa_2}}m(\mathcal{B})^{{\kappa_2}/{2p}}\Big)\\
&=\Big(\|\mathcal{F}\|_{M^{2p;\kappa_1}}\cdot\|\mathcal{G}\|_{M^{2p;\kappa_2}}\Big)\cdot m(\mathcal{B})^{\kappa/p},
\end{split}
\end{equation*}
since $\kappa={(\kappa_1+\kappa_2)}/2$, which in turn implies that the function $\mathcal{F}\cdot \mathcal{G}$ is in $M^{p;\kappa}(\mathbb R^n)$, and
\begin{equation}\label{holdermorrey}
\|\mathcal{F}\cdot \mathcal{G}\|_{M^{p;\kappa}}
\leq\|\mathcal{F}\|_{M^{2p;\kappa_1}}\cdot\|\mathcal{G}\|_{M^{2p;\kappa_2}}.
\end{equation}
Moreover, by using Theorem \ref{weak} with $q=2p$ and the inequality $2\sqrt{AB}\leq A+B$, $A,B>0$, we then have
\begin{equation*}
\begin{split}
\|\mathcal{F}\cdot \mathcal{G}\|_{M^{p;\kappa}}
&\lesssim p^2
\Big(\|\mathcal{F}\|_{LM^{p,r;\kappa_1}}^{1/2}\|\mathcal{F}\|_{\mathrm{BMO}}^{1/2}
\cdot\|\mathcal{G}\|_{LM^{p,r;\kappa_2}}^{1/2}\|\mathcal{G}\|_{\mathrm{BMO}}^{1/2}\Big)\\
&=p^2
\Big(\|\mathcal{F}\|_{LM^{p,r;\kappa_1}}\|\mathcal{G}\|_{\mathrm{BMO}}
\cdot\|\mathcal{G}\|_{LM^{p,r;\kappa_2}}\|\mathcal{F}\|_{\mathrm{BMO}}\Big)^{1/2}\\
&\lesssim p^2
\Big(\|\mathcal{F}\|_{LM^{p,r;\kappa_1}}\|\mathcal{G}\|_{\mathrm{BMO}}
+\|\mathcal{G}\|_{LM^{p,r;\kappa_2}}\|\mathcal{F}\|_{\mathrm{BMO}}\Big).
\end{split}
\end{equation*}
This proves \eqref{bilinear21}. In particular, by taking $r=p$, we get
\begin{equation}\label{bilinear22}
\|\mathcal{F}\cdot \mathcal{G}\|_{M^{p;\kappa}}\leq C(p,n)
\Big(\|\mathcal{F}\|_{M^{p;\kappa_1}}\|\mathcal{G}\|_{\mathrm{BMO}}
+\|\mathcal{G}\|_{M^{p;\kappa_2}}\|\mathcal{F}\|_{\mathrm{BMO}}\Big),
\end{equation}
for $\mathcal{F}\in M^{p;\kappa_1}(\mathbb R^n)\cap\mathrm{BMO}(\mathbb R^n)$ and $\mathcal{G}\in M^{p;\kappa_2}(\mathbb R^n)\cap\mathrm{BMO}(\mathbb R^n)$, when $1<p<\infty$ and $\kappa={(\kappa_1+\kappa_2)}/2$. Within the framework of Lorentz and Lorentz--Morrey spaces, these bilinear estimates \eqref{bilinear21}, \eqref{bilinear22} as well as \eqref{bilinear1} may be used in the study of certain problems (such as the global existence, uniqueness and regularity of weak solutions) in PDEs. In general, for the case $3\leq K\in \mathbb{N}$, we can also show that the following estimate
\begin{equation*}
\begin{split}
&\bigg\|\prod_{j=1}^K \mathcal{F}_j\bigg\|_{M^{p;\kappa}}\\
\lesssim& p^{K}
\sum_{j=1}^{K}\|\mathcal{F}_1\|_{\mathrm{BMO}}\cdots\|\mathcal{F}_{j-1}\|_{\mathrm{BMO}}
\|\mathcal{F}_j\|_{LM^{p,r;\kappa_j}}\|\mathcal{F}_{j+1}\|_{\mathrm{BMO}}\cdots\|\mathcal{F}_K\|_{\mathrm{BMO}}
\end{split}
\end{equation*}
holds for $\mathcal{F}_j\in LM^{p,r;\kappa_j}\cap\mathrm{BMO}(\mathbb R^n)$ with $1<p<\infty$, $p\leq r\leq\infty$, $0<\kappa_j<1$ and $\kappa=\sum_{j=1}^{K}\kappa_j/K$, $j=1,2,\dots,K$.

Let $p\geq1$ and $\Phi_p(t)$ be the same as before. As a direct consequence of Theorem \ref{weak}, we have the following result, which can be viewed as the John--Nirenberg type inequality in the setting of Lorentz--Morrey spaces.
For each fixed ball $\mathcal{B}$ in $\mathbb R^n$, there exists a positive constant $\widetilde{C}({n,p,r,\beta})$ depending on $n,p,r$ and $\beta$ such that
\begin{equation}\label{bilinear24}
\frac{1}{m(\mathcal{B})^{\kappa}}\int_{\mathcal{B}}\bigg[\Phi_p\Big(\beta\frac{|f(x)|}{\|f\|_{\mathrm{BMO}}}\Big)\bigg]\,dx
\leq \widetilde{C}({n,p,r,\beta})\bigg(\frac{\|f\|_{LM^{p,r;\kappa}}}{\|f\|_{\mathrm{BMO}}}\bigg)^p
\end{equation}
holds for all $f\in LM^{p,r;\kappa}(\mathbb R^n)\cap\mathrm{BMO}(\mathbb R^n)$ with $1<p<\infty$, $0<\kappa<1$ and $p\leq r\leq\infty$, provided that $0<\beta<\beta_n$, where $\beta_n$ is a positive constant which depends on $n,p$ and $r$.

Using the same strategy as in \eqref{bilinear14}, for any ball $\mathcal{B}$ in $\mathbb R^n$ and for $p\leq j<\infty$, we have that
\begin{equation}\label{wanguseinequ2}
\frac{1}{m(\mathcal{B})^{\kappa}}\int_{\mathcal{B}}|f(x)|^j\,dx
=\big\|f\big\|^j_{M^{j;\kappa}}\leq \widetilde{C}\big(n,p,r\big)^j
j^j\cdot\big(\|f\|_{LM^{p,r;\kappa}}\big)^p\cdot\big(\|f\|_{\mathrm{BMO}}\big)^{j-p}
\end{equation}
holds for all $f\in LM^{p,r;\kappa}(\mathbb R^n)\cap\mathrm{BMO}(\mathbb R^n)$, where $\widetilde{C}(n,p,r)$ is a positive constant depending on $n,r$ and $p$, but not on $j$. Therefore, changing the order of integration and summation and using the estimate \eqref{wanguseinequ2}, we obtain
\begin{equation*}
\begin{split}
&\frac{1}{m(\mathcal{B})^{\kappa}}\int_{\mathcal{B}}\bigg[\Phi_p\Big(\beta\frac{|f(x)|}{\|f\|_{\mathrm{BMO}}}\Big)\bigg]\,dx\\
&=\frac{1}{m(\mathcal{B})^{\kappa}}\int_{\mathcal{B}}
\sum_{j=p}^{\infty}\frac{\beta^j}{j!}\bigg(\frac{|f(x)|}{\|f\|_{\mathrm{BMO}}}\bigg)^jdx\\
&=\sum_{j=p}^{\infty}\frac{\beta^j}{j!}\cdot\frac{1}{\|f\|^j_{\mathrm{BMO}}}
\frac{1}{m(\mathcal{B})^{\kappa}}\int_{\mathcal{B}}|f(x)|^j\,dx\\
&\leq \sum_{j=p}^{\infty}\frac{\beta^j}{j!}\cdot\frac{1}{\|f\|^j_{\mathrm{BMO}}}\big[\widetilde{C}(n,p,r)j\big]^{j}
\big(\|f\|_{LM^{p,r;\kappa}}\big)^p\cdot\big(\|f\|_{\mathrm{BMO}}\big)^{j-p}\\
&=\sum_{j=p}^{\infty}\frac{[\widetilde{C}(n,p,r)\cdot\beta]^j\cdot j^{j}}{j!}
\bigg(\frac{\|f\|_{LM^{p,r;\kappa}}}{\|f\|_{\mathrm{BMO}}}\bigg)^p.
\end{split}
\end{equation*}
As in the proof of \eqref{bilinear14}, we also know that the above series is convergent provided $\widetilde{C}(n,p,r)\cdot\beta<e^{-1}$. That is, for
\begin{equation*}
0<\beta<\beta_n:=\big[\widetilde{C}(n,p,r)\cdot e\big]^{-1},
\end{equation*}
we see that \eqref{bilinear24} holds with $\widetilde{C}({n,p,r,\beta})$ is the sum of the above convergent series. In particular, by taking $r=p$ (in this special case, recall that the range of $p$ is $[1,\infty)$, as mentioned in Remark \ref{citeremark3}, and $\Phi_p(t)=\exp(t)-1$ when $p=1$), we have the following inequality
\begin{equation}\label{jninequality2}
\frac{1}{m(\mathcal{B})^{\kappa}}\int_{\mathcal{B}}
\bigg[\exp\Big(\beta\frac{|f(x)|}{\|f\|_{\mathrm{BMO}}}\Big)-1\bigg]\,dx
\leq \widetilde{C}({n,\beta})\bigg(\frac{\|f\|_{M^{1;\kappa}}}{\|f\|_{\mathrm{BMO}}}\bigg),
\end{equation}
whenever $f\in M^{1;\kappa}(\mathbb R^n)\cap\mathrm{BMO}(\mathbb R^n)$ with $0<\kappa<1$. A simple calculation gives
\begin{equation*}
\begin{split}
&\frac{1}{m(\mathcal{B})^{\kappa}}\int_{\mathcal{B}}
\bigg[\exp\Big(\beta\frac{|f(x)|}{\|f\|_{\mathrm{BMO}}}\Big)-1\bigg]\,dx\\
&=\frac{1}{m(\mathcal{B})^{\kappa}}\int_{\mathcal{B}}\sum_{j=1}^{\infty}
\frac{\beta^j}{j!}\bigg(\frac{|f(x)|}{\|f\|_{\mathrm{BMO}}}\bigg)^jdx\\
&=\sum_{j=1}^{\infty}\frac{\beta^j}{j!}\cdot\frac{1}{\|f\|^j_{\mathrm{BMO}}}
\frac{1}{m(\mathcal{B})^{\kappa}}\int_{\mathcal{B}}|f(x)|^j\,dx.
\end{split}
\end{equation*}
Furthermore, for any $\lambda>0$, we see that
\begin{equation*}
\begin{split}
\frac{1}{m(\mathcal{B})^{\kappa}}\int_{\mathcal{B}}|f(x)|^j\,dx
&\geq\frac{1}{m(\mathcal{B})^{\kappa}}\int_{\{x\in \mathcal{B}:|f(x)|>\lambda\}}|f(x)|^j\,dx\\
&>\frac{1}{m(\mathcal{B})^{\kappa}}\lambda^j\cdot m\big(\big\{x\in\mathcal{B}:|f(x)|>\lambda\big\}\big).
\end{split}
\end{equation*}
From this, it follows that
\begin{equation}\label{jninequality3}
\begin{split}
&\frac{1}{m(\mathcal{B})^{\kappa}}\int_{\mathcal{B}}
\bigg[\exp\Big(\beta\frac{|f(x)|}{\|f\|_{\mathrm{BMO}}}\Big)-1\bigg]\,dx\\
&\geq\frac{1}{m(\mathcal{B})^{\kappa}}
\sum_{j=1}^{\infty}\frac{\beta^j}{j!}\cdot\bigg(\frac{\lambda}{\|f\|_{\mathrm{BMO}}}\bigg)^j
\cdot m\big(\big\{x\in\mathcal{B}:|f(x)|>\lambda\big\}\big)\\
&=\frac{1}{m(\mathcal{B})^{\kappa}}
\Big[\exp\Big(\frac{\beta\lambda}{\|f\|_{\mathrm{BMO}}}\Big)-1\Big]
\cdot m\big(\big\{x\in\mathcal{B}:|f(x)|>\lambda\big\}\big).
\end{split}
\end{equation}
Hence, we conclude from \eqref{jninequality2} and \eqref{jninequality3} that for any ball $\mathcal{B}$ in $\mathbb R^n$ and for $\lambda>0$,
\begin{equation*}
\Big[\exp\Big(\frac{\beta\lambda}{\|f\|_{\mathrm{BMO}}}\Big)-1\Big]
\cdot\frac{m\big(\big\{x\in\mathcal{B}:|f(x)|>\lambda\big\}\big)}{m(\mathcal{B})^{\kappa}}
\leq \widetilde{C}({n,\beta})\bigg(\frac{\|f\|_{M^{1;\kappa}}}{\|f\|_{\mathrm{BMO}}}\bigg).
\end{equation*}
This in turn implies that
\begin{equation*}
\frac{m\big(\big\{x\in\mathcal{B}:|f(x)|>\lambda\big\}\big)}{m(\mathcal{B})^{\kappa}}
\leq \widetilde{C}({n,\beta})\bigg(\frac{\|f\|_{M^{1;\kappa}}}{\|f\|_{\mathrm{BMO}}}\bigg)
\cdot\Big[\exp\Big(\frac{\beta\lambda}{\|f\|_{\mathrm{BMO}}}\Big)-1\Big]^{-1}.
\end{equation*}
If $\lambda>\|f\|_{\mathrm{BMO}}$, then
\begin{equation*}
\Big[\exp\Big(\frac{\beta\lambda}{\|f\|_{\mathrm{BMO}}}\Big)-1\Big]^{-1}
\leq\exp\Big(-\frac{\beta\lambda}{\|f\|_{\mathrm{BMO}}}\Big)
\cdot\frac{1}{1-e^{-\beta}},
\end{equation*}
and hence
\begin{equation*}
\frac{m\big(\big\{x\in\mathcal{B}:|f(x)|>\lambda\big\}\big)}{m(\mathcal{B})^{\kappa}}
\leq \widetilde{C}({n,\beta})\bigg(\frac{\|f\|_{M^{1;\kappa}}}{\|f\|_{\mathrm{BMO}}}\bigg)
\exp\Big(-\frac{\beta\lambda}{\|f\|_{\mathrm{BMO}}}\Big),
\end{equation*}
whenever $\lambda>\|f\|_{\mathrm{BMO}}$. This interesting estimate is completely new(the John--Nirenberg type inequality in the setting of Morrey spaces).

\section{Appendix}
In the last section, we give the proof of Theorem \ref{appendix}. Suppose that $f\in LM^{p,\infty;\kappa_{*}}(\mathbb R^n)$ with $1<p<\infty$ and $0<\kappa_{*}<1$. We are going to prove that $f$ is in $LM^{q,r;\kappa}(\mathbb R^n)$ for $1\leq q<p<\infty$, $1\leq r<\infty$ and $0<\kappa<1$. Moreover,
\begin{equation}\label{desiredes}
\|f\|_{LM^{q,r;\kappa}}\leq \Big[\frac{pq}{r(p-q)}\Big]^{\frac{1}{r}}\|f\|_{LM^{p,\infty;\kappa_{*}}},
\end{equation} 
whenever $(1-\kappa)p=(1-\kappa_*)q$. To do this, for any fixed ball $\mathcal{B}$ in $\mathbb R^n$, we can write
\begin{equation*}
\begin{split}
&q\int_0^{\infty}\Big[\lambda d_{f;\mathcal{B}}(\lambda)^{1/q}\Big]^r\frac{d\lambda}{\lambda}\\
&=q\int_0^{\Lambda}\Big[\lambda d_{f;\mathcal{B}}(\lambda)^{1/q}\Big]^r\frac{d\lambda}{\lambda}
+q\int_{\Lambda}^{\infty}\Big[\lambda d_{f;\mathcal{B}}(\lambda)^{1/q}\Big]^r\frac{d\lambda}{\lambda}\\
&:=\mathrm{I+II},
\end{split}
\end{equation*}
where $\Lambda$ is a positive constant to be determined later. As for the estimate for the first term $\mathrm{I}$, it is clear that for all $\lambda>0$,
\begin{equation*}
d_{f;\mathcal{B}}(\lambda)=m\big(\big\{x\in \mathcal{B}:|f(x)|>\lambda\big\}\big)\leq m(\mathcal{B}).
\end{equation*}
Then we have
\begin{equation*}
\mathrm{I}\leq m(\mathcal{B})^{\frac{r}{q}}q\int_0^{\Lambda}\lambda^{r-1}d\lambda
=m(\mathcal{B})^{\frac{r}{q}}\frac{q}{\,r\,}\cdot\Lambda^r.
\end{equation*}
Let us now turn to the estimate of the second term $\mathrm{II}$. By Proposition \ref{prop21}, we know that for any ball $\mathcal{B}$ in $\mathbb R^n$,
\begin{equation*}
\|f\|_{L^{p,\infty}(\mathcal{B})}=\sup_{s>0}\Big[s^{1/p}f^{*}_{\mathcal{B}}(s)\Big]
=\sup_{\lambda>0}\Big[\lambda d_{f;\mathcal{B}}(\lambda)^{1/p}\Big],
\end{equation*}
and hence
\begin{equation*}
\begin{split}
\mathrm{II}&\leq q\int_{\Lambda}^{\infty}
\Big[\lambda\Big(\frac{\|f\|_{L^{p,\infty}(\mathcal{B})}}{\lambda}\Big)^{\frac{p}{q}}\Big]^r\frac{d\lambda}{\lambda}\\
&=\|f\|_{L^{p,\infty}(\mathcal{B})}^{\frac{pr}{q}}\frac{q^2}{r(p-q)}\cdot\Lambda^{r-\frac{rp}{q}}.
\end{split}
\end{equation*}
Combining the estimates for $\mathrm{I}$ and $\mathrm{II}$ leads to
\begin{equation*}
q\int_0^{\infty}\Big[\lambda d_{f;\mathcal{B}}(\lambda)^{1/q}\Big]^r\frac{d\lambda}{\lambda}
\leq m(\mathcal{B})^{\frac{r}{q}}\frac{q}{\,r\,}\cdot\Lambda^r
+\|f\|_{L^{p,\infty}(\mathcal{B})}^{\frac{pr}{q}}\frac{q^2}{r(p-q)}\cdot\Lambda^{r-\frac{rp}{q}}.
\end{equation*}
We now choose $\Lambda>0$ so that
\begin{equation*}
m(\mathcal{B})^{\frac{r}{q}}\cdot\Lambda^r=\|f\|_{L^{p,\infty}(\mathcal{B})}^{\frac{pr}{q}}\cdot\Lambda^{r-\frac{rp}{q}},
\end{equation*}
namely,
\begin{equation*}
\Lambda:=\frac{\|f\|_{L^{p,\infty}(\mathcal{B})}}{m(\mathcal{B})^{\frac{1}{p}}}.
\end{equation*}
Therefore,
\begin{equation*}
\begin{split}
q\int_0^{\infty}\Big[\lambda d_{f;\mathcal{B}}(\lambda)^{1/q}\Big]^r\frac{d\lambda}{\lambda}
&\leq \frac{q}{\,r\,}\Big(1+\frac{q}{p-q}\Big)\|f\|_{L^{p,\infty}(\mathcal{B})}^{\frac{pr}{q}}\cdot\Lambda^{r-\frac{rp}{q}}\\
&=\frac{q}{\,r\,}\cdot\frac{p}{p-q}
\cdot\frac{\|f\|^r_{L^{p,\infty}(\mathcal{B})}}{m(\mathcal{B})^{\frac{r}{p}-\frac{r}{q}}}.
\end{split}
\end{equation*}
This, together with the definition of $LM^{p,\infty;\kappa_{*}}(\mathbb R^n)$, implies that
\begin{equation*}
\begin{split}
\bigg(q\int_0^{\infty}\Big[\lambda d_{f;\mathcal{B}}(\lambda)^{1/q}\Big]^r\frac{d\lambda}{\lambda}\bigg)^{\frac{1}{r}}
&\leq\Big[\frac{pq}{r(p-q)}\Big]^{\frac{1}{r}}
\cdot\frac{\|f\|_{L^{p,\infty}(\mathcal{B})}}{m(\mathcal{B})^{\frac{1}{p}-\frac{1}{q}}}\\
&\leq\Big[\frac{pq}{r(p-q)}\Big]^{\frac{1}{r}}
\frac{m(\mathcal{B})^{\frac{\kappa_{*}}{p}}}{m(\mathcal{B})^{\frac{1}{p}-\frac{1}{q}}}\|f\|_{LM^{p,\infty;\kappa_{*}}}.
\end{split}
\end{equation*}
Finally, by taking the supremum over all balls $\mathcal{B}\subset\mathbb R^n$ and using Proposition \ref{prop24}, we obtain the desired estimate \eqref{desiredes}, whenever
\begin{equation*}
\kappa_*=1-\frac{(1-\kappa)p}{q}\Longleftrightarrow(1-\kappa)p=(1-\kappa_*)q.
\end{equation*}  
To summarize, we have the following inclusion relations
\begin{equation*}
LM^{p,r;\kappa_{*}}(\mathbb R^n)\subset LM^{p,\infty;\kappa_{*}}(\mathbb R^n)\subset LM^{q,r;\kappa}(\mathbb R^n),
\end{equation*}
provided that $1\leq q<p<\infty$, $1\leq r<\infty$, $0<\kappa_{*},\kappa<1$ and $(1-\kappa)p=(1-\kappa_*)q$. Moreover, the size of the embedding constant from $LM^{p,\infty;\kappa_{*}}(\mathbb R^n)$ into $LM^{q,r;\kappa}(\mathbb R^n)$ is specified, the difference between some Lorentz--Morrey spaces is clarified. In particular, when $r=q$,
\begin{equation*}
M^{p;\kappa_{*}}(\mathbb R^n)\subset WM^{p;\kappa_{*}}(\mathbb R^n)\subset M^{q;\kappa}(\mathbb R^n),
\end{equation*}
provided that $1\leq q<p<\infty$, $0<\kappa_{*},\kappa<1$ and $(1-\kappa)p=(1-\kappa_*)q$. Moreover,
\begin{equation*}
\|f\|_{M^{q;\kappa}}\leq \Big[\frac{p}{p-q}\Big]^{\frac{1}{q}}\|f\|_{WM^{p;\kappa_{*}}}.
\end{equation*}

\end{document}